\def\arxivmode{}

\ifdefined\arxivmode
\documentclass[11pt]{article}
\usepackage[margin=1in]{geometry}
\usepackage{amsmath,amssymb,amsfonts,graphicx,color,array}
\usepackage{amsthm}
\usepackage{tgtermes}
\usepackage{newtxtext}
\usepackage{newtxmath}
\usepackage{bm}
\usepackage{endnotes}

\newcommand{\OneAndAHalfSpacedXII}{}

\newcommand{\EquationsNumberedThrough}{}

\newcommand{\TheoremsNumberedThrough}{}

\newcommand{\ECRepeatTheorems}{}
\newcommand{\RUNAUTHOR}[1]{}
\newcommand{\RUNTITLE}[1]{}
\newcommand{\MANUSCRIPTNO}[1]{}
\newcommand{\HISTORY}[1]{}

\makeatletter
\newcommand{\arxiv@title}{}
\newcommand{\TITLE}[1]{\renewcommand{\arxiv@title}{#1}}
\newcommand{\arxiv@authors}{}
\newcommand{\ARTICLEAUTHORS}[1]{\renewcommand{\arxiv@authors}{#1}}
\newcommand{\AUTHOR}[1]{\par\medskip\textbf{#1}}
\newcommand{\AFF}[1]{\par{\small\itshape #1}}
\newcommand{\EMAIL}[1]{\texttt{#1}}

\newcommand{\arxiv@abstract}{}
\newcommand{\ABSTRACT}[1]{\renewcommand{\arxiv@abstract}{#1}}
\newcommand{\arxiv@keywords}{}
\newcommand{\KEYWORDS}[1]{\renewcommand{\arxiv@keywords}{#1}}
\newcommand{\arxiv@funding}{}
\newcommand{\FUNDING}[1]{\renewcommand{\arxiv@funding}{#1}}
\renewcommand{\maketitle}{%
  \begin{center}
    {\LARGE\bfseries\arxiv@title\par}\vspace{1em}
    \arxiv@authors
  \end{center}\vspace{1em}
  \begin{abstract}\arxiv@abstract\end{abstract}
  \noindent\textbf{Keywords:} \arxiv@keywords\par\medskip
  \noindent\textbf{Funding:} \arxiv@funding\par\bigskip}
\makeatother

\renewenvironment{proof}[1]{\par\noindent\textit{#1}\enskip\ignorespaces}{\par\addvspace{0pt}}
\newcommand{\Halmos}{\mbox{\quad$\square$}}

\long\def\FIGURE#1#2#3{\centering #1\caption{#2}}
\long\def\TABLE#1#2#3{\caption{#1}\centering #2}

\newcommand{\SingleSpacedXI}{}
\newcommand{\SingleSpacedXII}{}

\def\argmin{\mathop{\rm arg\,min}}

\newenvironment{APPENDICES}{\appendix}{}

\else
\documentclass[ijoc,sglanonrev]{informs4}
\RequirePackage{tgtermes}
\RequirePackage{newtxtext}
\RequirePackage{newtxmath}
\RequirePackage{bm}
\RequirePackage{endnotes}

\OneAndAHalfSpacedXII % Current default line spacing
\fi

\usepackage{mathtools}
\usepackage{tabularx}
\usepackage{lscape}
\usepackage{cases}
\usepackage{booktabs}
\usepackage{multirow}
\usepackage{colortbl,hhline}
\usepackage{tikz}
\usepackage{setspace}
\usepackage[caption=false]{subfig}
\usetikzlibrary{calc,arrows.meta,bending, positioning,decorations.pathreplacing,calligraphy,fit,patterns.meta}
\pgfmathdeclarerandomlist{RandBlackWhite}{%
	{black}%
		{white}%
}
\pgfmathsetseed{2}
\usepackage{cancel}
\usepackage{xcolor}
\usepackage[normalem]{ulem}
\usepackage[ruled, vlined,linesnumbered]{algorithm2e}
\usepackage{enumerate}
\usepackage{url}
\usepackage{natbib}
 \bibpunct[, ]{(}{)}{,}{a}{}{,}%
 \def\BIBand{and}%

 \mathchardef\mhyphen="2D

\newcommand{\Tmax}{T_{\max}}
\newcommand{\RangeOper}[1]{\llbracket #1 \rrbracket}
\newcommand{\Trange}{\RangeOper{T}}
\newcommand{\Tmaxrange}{\RangeOper{\Tmax}}
\newcommand{\B}{\mathbb{B}}

\newcommand{\vbd}{\mathbf{d}}
\newcommand{\vbx}{\mathbf{x}}
\newcommand{\vby}{\mathbf{y}}
\newcommand{\vbz}{\mathbf{z}}
\newcommand{\vbw}{\mathbf{w}}
\newcommand{\vbu}{\mathbf{u}}
\newcommand{\vbh}{\mathbf{h}}

\newcommand{\phiIndicator}{p}
\newcommand{\vbzero}{\mathbf{0}}
\newcommand{\vbone}{\mathbf{1}}

\newcommand{\ControlSet}{\mathcal{D}^{\mathrm{all}}}
\newcommand{\PrevControlSet}{\mathcal{D}}
\newcommand{\TargetSize}{\lambda}
\newcommand{\TargetSizeMax}{\lambda_\text{max}}
\newcommand{\GeneSet}{\mathcal{I}}
\newcommand{\CtrlGeneSet}{\mathcal{J}}
\newcommand{\UnctrlGeneSet}{\mathcal{J}^{\mathrm{c}}}

\newcommand{\ULP}{\mathcal{U}}
\newcommand{\LLP}{\mathcal{L}}
\newcommand{\SubProb}{\mathcal{S}}
\newcommand{\MP}{\mathcal{M}}
\newcommand{\CutSet}{\mathcal{V}}
\newcommand{\ULPModel}{\ULP(\TargetSize,\PrevControlSet)}
\newcommand{\SubspaceP}{\SubProb^*}
\newcommand{\SubproblemAt}[2]{\SubProb^{#2}(#1)}
\newcommand{\SubproblemModelAtT}{\SubproblemAt{\vbd}{T}}
\newcommand{\LLPModel}{\LLP^{\Tmax}(\vbd)}
\newcommand{\LLPModelAtT}[1]{\LLP^{#1}(\vbd)}
\newcommand{\SubspaceModel}{\SubspaceP(\vbd)}
\newcommand{\MPModel}{\MP(\TargetSize,\PrevControlSet,\CutSet)}

\newcommand{\UseTScut}{\texttt{UseTScut}}
\newcommand{\phigene}{\varphi}
\newcommand{\phiFunction}{\phi}

\usepackage{xr-hyper}
\usepackage{hyperref}
\usepackage[capitalise]{cleveref}
\crefname{figure}{Figure}{Figures}
\Crefname{figure}{Figure}{Figures}
\crefname{constraint}{Constraint}{Constraints}
\Crefname{constraint}{Constraint}{Constraints}
\creflabelformat{constraint}{(#2#1#3)}
\crefname{objective}{Objective}{Objectives}
\Crefname{objective}{Objective}{Objectives}
\crefname{appendix}{Appendix}{Appendices}
\Crefname{appendix}{Appendix}{Appendices}
\creflabelformat{objective}{(#2#1#3)}
\crefname{definition}{Definition}{Definitions}
\Crefname{definition}{Definition}{Definitions}
\crefname{equation}{Eq.}{Eqs.}
\Crefname{equation}{Eq.}{Eqs.}
\creflabelformat{equation}{(#2#1#3)}

\ifdefined\arxivmode
\theoremstyle{plain}
\newtheorem{theorem}{Theorem}

\theoremstyle{definition}
\newtheorem{definition}[theorem]{Definition}

\fi

\definecolor{msored}{HTML}{C00000}

\makeatletter
\newcommand*{\addFileDependency}[1]{% argument=file name and extension
  \typeout{(#1)}
  \@addtofilelist{#1}
  \IfFileExists{#1}{}{\typeout{No file #1.}}
}
\makeatother

\newcommand{\alphaj}{\alpha_j}
\newcommand{\betaj}{\beta_j}
\newcommand{\kj}{k_j}

\newcommand{\SACP}[1]{\textbf{SACP}(#1)}
\newcommand{\TmaxSACP}{\SACP{$\Tmax$}}
\newcommand{\InfiniteSACP}{\SACP{$\infty$}}

\newcommand{\numgenes}{|\GeneSet|}
\newcommand{\numctrlgenes}{|\CtrlGeneSet|}
\newcommand{\ClauseSet}{\mathfrak{C}}
\newcommand{\numclauses}{|\ClauseSet^1|}

\newcommand{\tscut}{TS cut}
\newcommand{\attrcut}{AT cut}

\newcommand{\SEP}{\texttt{SEP}}
\newcommand{\BEN}{\texttt{BEN}}
\newcommand{\PBN}{\texttt{PBN}}
\newcommand{\MibS}{\texttt{MibS}}

\newcolumntype{L}[1]{>{\raggedright\let\newline\\\arraybackslash\hspace{0pt}}m{#1}}
\newcolumntype{C}[1]{>{\centering\let\newline\\\arraybackslash\hspace{0pt}}m{#1}}
\newcolumntype{R}[1]{>{\raggedleft\let\newline\\\arraybackslash\hspace{0pt}}m{#1}}
\newcolumntype{E}{>{\raggedleft\let\newline\\\arraybackslash\hspace{0pt}}m{0cm}}

\definecolor{msoblue}{RGB}{0, 112, 192}

\NewDocumentCommand{\soutRobust}{m m}{
  {%
    \let\oldcite\cite
    \let\oldcitep\citep
	\let\oldcreff\cref
    \renewcommand{\cite}[1]{\mbox{\oldcite{##1}}}
    \renewcommand{\citep}[1]{\mbox{\oldcitep{##1}}}
	\renewcommand{\cref}[1]{\mbox{\oldcreff{##1}}}
    \textcolor{#2}{\sout{#1}}%
    \let\cite\oldcite
    \let\citep\oldcitep
	\let\cref\oldcreff
  }
}

\NewDocumentCommand{\kds}{m}{
  \soutRobust{#1}{blue}
}

\makeatletter
\def\AAnormalsizeXII{%
  \@setfontsize\normalsizeXI{12}{19.7}%
  \Dskips{1.0}{0.3}{0.5}{0.6}%
}
\allowdisplaybreaks[4]
\AtBeginDocument{\interdisplaylinepenalty=0}

\renewcommand{\kds}[1]{}
\EquationsNumberedThrough    % Default: (1), (2), ...
\TheoremsNumberedThrough     % Preferred (Theorem 1, Lemma 1, Theorem 2)
\ECRepeatTheorems  %  

\MANUSCRIPTNO{}

\begin{document}
%%%%%%%%%%%%%%%%

%\setcounter{page}{1}
%\DOI{10.1287/xxxx.0000.0000}%

\RUNAUTHOR{Moon, Lee, and Paulevé}

% \RUNTITLE{Predictive Maintenance in Manufacturing}
\RUNTITLE{Bilevel IP for the synchronous attractor control problem}

% \TITLE{Optimal Resource Allocation in Humanitarian Logistics: A Stochastic Programming Approach}
\TITLE{A Bilevel Integer Programming Approach for the Synchronous Attractor Control Problem}

\ARTICLEAUTHORS{%
\AUTHOR{Kyungduk Moon, Kangbok Lee}
	\AFF{Department of Industrial and Management Engineering, POSTECH, South Korea\\ \EMAIL{kaleb.moon@postech.ac.kr}, \EMAIL{kblee@postech.ac.kr}} %, \URL{}
	\AUTHOR{Loïc Paulevé}
	\AFF{Univ. Bordeaux, CNRS, Bordeaux INP, LaBRI, UMR 5800 F-33400 Talence, France\\ \EMAIL{loic.pauleve@labri.fr}}
} % end of the block

\ABSTRACT{%
	Boolean networks are dynamical models of disease development in which the activation levels of genes are represented by binary variables. Given a Boolean network, controls represent mutations or medical treatments that fix the activation levels of selected genes so that all states in every attractor (i.e., long-term recurrent states) satisfy a desired phenotype. Our goal is to enumerate all minimal controls, identifying critical gene subsets in disease development and therapy. This problem has an inherent bilevel integer programming structure and is computationally challenging.

	We propose an infeasibility-based Benders decomposition, a logic-based Benders framework for bilevel integer programs with multiple subproblems. In our application, each subproblem finds a forbidden attractor of a given length and yields a problem-specific feasibility cut. We also propose an auxiliary IP called subspace separation that finds a Boolean subspace that includes multiple forbidden attractors and thereby strengthens the cut. Numerical experiments show that the resulting algorithms are much more scalable than state-of-the-art methods and that subspace separation substantially improves performance.
}%

\FUNDING{
	The work of KM and KL was supported under the framework of international cooperation program managed by National Research Foundation of Korea (No. RS-2023-00259481), and supported by Basic Science Research Program through the National Research Foundation of Korea funded by the Ministry of Education (No. RS-2025-25433718).
	Work of LP was supported by Campus France in the scope of the bilateral France-South Korea PHC STAR project number 50147NM and by the French Agence Nationale pour la Recherche (ANR) in the scope of the
	ANR-NRF project ``REPAIRNET'' (grant number ANR-25-IABT-0003).}

%Supplemental Material:
%Data Ethics & Reproducibility Note:

%\KEYWORDS{Stochastic programming, Decision support,Uncertainty, Disaster response, Optimization}

%\HISTORY{Received: Month DD, YYYY; Accepted: Month DD, YYYY; Published Online: Month DD, YYYY}

\maketitle
%%%%%%%%%%%%%%%%%%%%%%%%%%%%%%%%%%%%%%%%%%%%%%%%%%%%%%%%%%%%%%%%%%%%%%

% Text of your paper here

\clearpage
\section{Introduction} \label{sec:introduction}
\subsection{Precision Medicine and Boolean Networks}
\textit{Precision medicine} is a paradigm that aims to provide tailored treatments to individuals based on their susceptibility and response to disease. The development of disease is regulated by interactions among many genes, biological molecules, and environmental factors; we refer to all these factors simply as `genes'.
An essential task of precision medicine is to identify combinations of
genetic alterations that either promote disease development or facilitate recovery \citep{lord_2017_PARP,vargas2016biomarker}. Because this process requires extensive experiments, an efficient computational tool is crucial to narrowing down the set of alterations.

Recent studies have utilized \textit{Boolean networks} (BNs) to model the development of diseases, demonstrating their practical relevance \citep{schwab_2020_Concepts}. A BN is a qualitative, discrete dynamical model that represents the activation levels of genes in a set $\GeneSet$ as binary variables.
Mathematically, it is a Boolean map $f:\{0,1\}^{\numgenes} \to\{0,1\}^{\numgenes}$ between binary vectors of dimension $\numgenes$ where 0 and 1 represent the low and the high activation levels, respectively. Interactions among genes are represented as Boolean formulas, referred to as \textit{transition formulas}. 
To construct a dynamical system, 
a \textit{state} is an assignment of Boolean values to genes where $2^{\numgenes}$ states are possible.
Given a state $\vbx$, the \emph{synchronous update mode} updates all variables together as the image of the state by $f$ (i.e., $\vbx$ becomes $f(\vbx)$). This mode has proven effective for modeling biological systems \citep{mori_2022_Attractor}; other possible update modes are summarized by \cite{pauleve_2021_Boolean}, but we leave them for future work.

BNs offer several advantages over alternative modeling approaches. Unlike machine-learning models, they explicitly represent the disease dynamics and causal interactions among genes. They are also more computationally efficient and scalable than continuous dynamical models such as differential equations.

We are interested in the long-term recurrent sets of states, called \textit{attractors}.
An attractor is a mutually reachable set of states where no transition leaves the set \citep{kauffman1993origins}.
For the synchronous update, state $\vbx$ is in an attractor if and only if there exists a natural number $T$ such that $\vbx=f^T(\vbx)$. The smallest such number $T$ is called the \emph{length},
and the set of states $\{f(\vbx), f^2(\vbx), \ldots, f^T(\vbx)\}$ constitutes the attractor.
Attractors are considered to determine the long-term behavior of the dynamical system because every initial state eventually reaches one of the attractors.
A relevant concept is a \textit{trap space}, a Boolean subspace where certain variables are fixed at either 0 or 1, and no transition leaves the subspace \citep{klarner_2015_Computing}. Trap spaces are important in this paper because every trap space includes one or more attractors as a subset, providing a tractable over-approximation of attractors.

The biological or medical condition of each state can be evaluated by another Boolean function representing the \textit{phenotype}, denoted by $\phiFunction:\{0,1\}^{\numgenes}\to\{0,1\}$. Depending on the purpose of modeling, it may represent either the diseased condition or the normal condition, as inferred from accumulated experiments and clinical trials. If every state in all attractors satisfies the disease phenotype, this strongly suggests that the disease will develop. This analysis has been widely used to identify treatments for bladder cancer \citep{grieco_2013_Integrative,remy_2015_Modeling}, breast cancer \citep{biane_2019_Causal}, and so on.

\newcommand{\DrawColoredExampleNodes}{
	% the lower plane 
	\draw[] node at (000) {$000$};
	\draw[] node at (001) {$001$};
	\draw[] node at (010) {$010$};
	\draw[] node at (100) {$100$};
	\draw[] node at (101) {$101$};
	\draw[] node at (110) {$110$};
	\draw[msored] node at (011) {$011$};
	\draw[msored] node at (111) {$111$};
}

\newcommand{\DrawExampleNodes}{
	% the lower plane 
	\draw[] (0,0) node(101) {101};
	\draw[] (3,0) node(111) {111};
	\draw[] (2 + 0.2, 1 - 0.2) node(011) {011};
	\draw[] (1 - 0.2, 1 - 0.2) node(001) {001};

	% the upper plane
	\draw[] (0,3) node(100) {100};
	\draw[] (1 - 0.2, 2 + 0.2) node(000) {000};
	\draw[] (2 + 0.2, 2 + 0.2) node(010) {010};
	\draw[] (3,3) node(110) {110};
}

\tikzstyle{attractor_fill} = [rectangle,draw,fill=black!20,minimum height=1.4em,minimum width=1.8em]
\tikzstyle{normal_arc} = [-{stealth[flex=0.75]},black]
\tikzstyle{thick_arc} = [-{stealth[flex=0.75]},black,very thick]
\def\exgapEps{0.09}
\def\exgaphalf{0.2}
\def\exgap{0.4}
\def\exgapL{0.5}
\def\exgapLL{0.6}
\def\exgapLLL{0.8}

\newcommand{\DrawEmphAttractors}{
	% the lower plane 
	\node at (0,0) [attractor_fill] {};
	\node at (3,0) [attractor_fill] {};
	\node at (2.2,0.8) [attractor_fill] {};
	\node at (0.8,2.2) [attractor_fill] {};
	% \node at (0.8,2.2) [densely dashed,rectangle,gray,thick, draw,minimum height=2.5em,minimum width=3.0em] {};
}
\newcommand{\DrawExampleArcs}{
	\draw[normal_arc] (100) ..  controls +(down:0.5) and +(left:1) .. (001);
	\draw[normal_arc] (010) -- (001);
	\draw[normal_arc] (001) -- (101);
	\draw[thick_arc] (011) -- (101);
	\draw[thick_arc] (101) -- (111);
	\draw[thick_arc] (111) -- (011);
	\draw[normal_arc] (110) ..  controls +(down:0.5) and +(right:1) ..  (001);
	\draw[thick_arc] (000) +(up:\exgap) arc (-100:230:2.5mm);
}

\newcommand{\DrawEmphNewAttractors}{
	% the lower plane 
	\node at (3,0) [attractor_fill] {};
	\node at (2.2,0.8) [attractor_fill] {};
}

\newcommand{\DrawTrapSpace}[1]{
	\node at (0.8,2.2) [densely dashed,rectangle,gray,thick, draw,minimum height=2.5em,minimum width=2.9em] {};

	\draw[densely dashed,rectangle,gray,thick] (0-\exgapL,0-\exgapLL) rectangle (3+\exgapL,1+\exgapL);
	\draw[#1] (0-\exgapLLL,0-\exgapLLL) rectangle (3+\exgapLLL,3+\exgapLLL);

}
\newcommand{\DrawNewTrapSpace}[1]{
	% \draw[#1] (2-\exgapL,3+\exgap) -- (3+\exgapL,3+\exgap)  -- (3+\exgapL,0-\exgap) -- (2-\exgapL,0-\exgap) -- cycle;
	\draw[densely dashed,orange,thick] (2-\exgapL,0-\exgapL) rectangle (3+\exgapLL,3+\exgapL);
	\draw[#1] (0-\exgapL,0-\exgapLL) rectangle (3+\exgapL,1+\exgapL);
	\draw[#1] (0-\exgapLLL,0-\exgapLLL) rectangle (3+\exgapLLL,3+\exgapLLL);
	\draw[densely dashed,gray,thick] (2-\exgapL,0-\exgapL) +(\exgapEps,\exgapEps) rectangle ([shift={(-\exgapEps,-\exgapEps)}] 3+\exgapL,1+\exgapL);
}

\tikzset{legendBoxStyle/.style={draw,minimum width=1.8cm, minimum height=0.8cm}}
\newcommand{\STGLegend}{
	\begin{tikzpicture}
		\SingleSpacedXII
		\small
		
		\begin{scope}[local bounding box=legendBox]
			
			% Attractor
			\node[legendBoxStyle, fill=gray!30, text=black] at (0,0) {Attractor};

			% Trap space
			\node[legendBoxStyle, dashed, gray, text=black] at (0,-1) {Trap space};

			% Control
			\node[legendBoxStyle, dashed, orange, text=black] at (0,-2) {Control};
			
		\end{scope}

		% Surrounding gray box
		% \node[draw, gray!50, fit=(legendBox), inner sep=8pt] {};
		
	\end{tikzpicture}
}

\newcommand{\DrawExampleArcsSync}[1]{
	\ifnum #1 = 1
		\def\thickarc{thick_arc}
		\def\thickarcTwo{thick_arc_two}
	\else
		\def\thickarc{normal_arc}
		\def\thickarcTwo{normal_arc}
	\fi
	\draw[normal_arc] (100) ..  controls +(down:0.5) and +(left:1) .. (001);
	\draw[normal_arc] (010) -- (001);
	\draw[normal_arc] (001) -- (101);
	\draw[\thickarc] (011) -- (101);
	\draw[\thickarc] (101) -- (111);
	\draw[\thickarc] (111) -- (011);
	\draw[normal_arc] (110) ..  controls +(down:0.5) and +(right:1) ..  (001);

	% \draw[\thickarc] (111) ..  controls +(up:0.8) and +(right:0.8) ..   (011);

	% \draw[\thickarcTwo] (000) +(up:\exgap) arc (-100:240:1.5mm);
}

\begin{figure}[t]
	
	\FIGURE{
	\subfloat[Inputs\label{subfig:example-transition-formula}]{%
		\parbox[b]{0.26\linewidth}{
			\centering
			\SingleSpacedXI
			\small
			\begin{tabular}{r@{}l}
				\toprule
				\multicolumn{2}{l}{\textit{\textbf{The transition formulas}}}                 \\
				\midrule
				$f_1(\vbx)$ & \, $= (\neg x_1 \vee \neg x_2) \wedge x_3$ \\
				$f_2(\vbx)$ & \, $= x_1 \wedge x_3$                      \\
				$f_3(\vbx)$ & \, $= x_1 \vee x_2 \vee x_3$              \\
				\midrule
				\multicolumn{2}{l}{\textit{\textbf{The phenotype}}}                   \\
				\midrule
				$\phiFunction(\vbx)$ & \, $= x_2 \wedge x_3$                    \\
				\bottomrule
			\end{tabular}
			\vspace{0pt}
		}
	}\hfill
	\subfloat[STG (no control)\label{subfig:example-state-transition-graph}]{%
		\parbox[b]{0.28\linewidth}{
			\centering
			\begin{tikzpicture}[scale=0.9]
				\DrawEmphAttractors
				\DrawExampleNodes
				\DrawExampleArcsSync{1}
				\DrawColoredExampleNodes{}
				\DrawTrapSpace{densely dashed,gray,thick}
			\end{tikzpicture}
		}
	}\hfill
	\subfloat[STG (control fixing $x_2=1$)\label{subfig:example-state-transition-graph-control}]{%
		\parbox[b]{0.28\linewidth}{
			\centering
			\begin{tikzpicture}[scale=0.9]
				\DrawEmphNewAttractors
				\DrawExampleNodes
				\DrawNewTrapSpace{densely dashed,gray,thick}
				\draw[normal_arc] (100) .. controls +(down:0.5) and +(left:1) .. (011);
				\draw[normal_arc] (000) -- (010);
				\draw[normal_arc] (010) -- (011);
				\draw[normal_arc] (001) -- (111);
				\draw[thick_arc] (011) .. controls +(right:0.8) and +(up:0.8) .. (111);
				\draw[normal_arc] (101) -- (111);
				\draw[thick_arc] (111) .. controls +(up:0.8) and +(right:0.8) .. (011);
				\draw[normal_arc] (110) .. controls +(down:0.5) and +(north east:1) .. (011);
			\end{tikzpicture}
		}
	}\hfill
	\begin{minipage}[b]{0.12\linewidth}
		\vspace{0pt}
		\small
		\centering
		\STGLegend{}
	\end{minipage}
	}
	{A Boolean network example and state transition graphs (STGs)\label{fig:example}}{}
\end{figure}
\Cref{fig:example} shows an example of a BN with three genes denoted as
1--3. \Cref{fig:example}\protect\subref{subfig:example-transition-formula} presents the
transition formulas. \Cref{fig:example}\protect\subref{subfig:example-state-transition-graph} shows the corresponding state transition graph (STG) under the synchronous update. The example presents two attractors $\{(000)\}$ and $\{(101),(111),(011)\}$. Ideally, all attractor states would satisfy the phenotype, but this does not hold for states $(000)$ and $(101)$.
\Cref{fig:example}\protect\subref{subfig:example-state-transition-graph} also illustrates two trap spaces that can be denoted as $(000)$ and $(**1)$, where 0 or 1 implies the corresponding variable is fixed to that value and $*$ implies it is not fixed. 

\subsection{The Synchronous Attractor Control Problem (SACP)}

In a BN, \emph{control} means permanently fixing the transition formula of some genes to either 0 or 1, and their activation levels accordingly \citep{biane_2026_Why}.
A control may change the set of attractors, as well as the phenotype compositions of the states in attractors. Biologists consider controls fixing fewer genes more relevant, as they highlight the parsimonious causes of a disease and key targets for treatment \citep{biane_2019_Causal}.
At the same time, more candidates are desirable to validate with detailed
experiments. Hence, the main objective is to enumerate all \textit{minimal controls}, defined as follows.
\begin{definition}[Minimal control]
A control is \emph{feasible} if all states of every attractor under the control satisfy the phenotype. A feasible control is \emph{minimal} if it is inclusion-wise minimal, i.e., unfixing any single gene would violate feasibility.	
\end{definition}
As a recent survey of control problems in BNs shows \citep{biane_2026_Why}, enumeration of minimal controls is a common goal for more than 16 existing tools.

\Cref{fig:example}\protect\subref{subfig:example-state-transition-graph-control} illustrates the
new state transition graph of \Cref{fig:example}\protect\subref{subfig:example-transition-formula} under
the control fixing $x_2=1$, overriding the original transition formula $f_2(\vbx)=x_1 \wedge x_3$. The previous attractor $\{(000)\}$, which does not satisfy the phenotype, completely
vanishes. As a result, all states of the remaining attractor $\{(011),(111)\}$ satisfy the phenotype.
The set of trap spaces also changes by the control; $(000)$ is removed, while $(*11)$ and $(*1*)$ are introduced.
Applying $x_3=1$ in addition to $x_2=1$ is also feasible, but we will not accept it as a minimal control because $x_2=1$ alone is feasible.

Let $\Tmax\in \mathbb{Z}^{+}$ be the maximum attractor length considered:
attractors longer than $\Tmax$ will be discarded when evaluating the
phenotype. We will refer to the 
enumeration problem of minimal controls as
\textit{the synchronous attractor control problem with the maximum length $\Tmax$}, denoted as \TmaxSACP{}.
For example, the problem limited to length one can be denoted as \SACP{1}, and the unrestricted problem is denoted as \InfiniteSACP{} ($\equiv \SACP{2^{\numgenes}}$).

\def\tabHeader{\toprule
     \textbf{Reference} & $\Tmax$ & \textbf{Task} & \textbf{Algorithm} \\ \midrule
}
\begin{table}[t]
	\caption{Literature on the Synchronous Attractor Control Problem.}\label{tab:literature-review}
	\SingleSpacedXI
\centering
\begin{tabular}{llll}
	\tabHeader
      %%%
      %%%                                                             
 \citet{hayashida_2008_Algorithms}                          &      1  & Search & \textbf{Brute-force} (exhaustive)                                 \\
 \citet{akutsu_2012_Integer}; \citet{qiu_2014_Control}      &      1  & Search & \textbf{IP} (alternating two models)                      \\
 \citet{murrugarra_2016_Identification}                     &      1  & Enumeration & \textbf{Symbolic} (Gr\"obner basis)                            \\
 \citet{biane_2019_Causal}                &      1  & Enumeration & \textbf{Symbolic} (binary decision diagram)        \\
 \citet{moon_2022_Bilevel}               &      1  & Enumeration & \textbf{Bilevel IP} (branch-and-bound) \\
 \midrule
 \citet{cifuentesfontanals_2020_Control}                    &      $\infty$ & Enumeration &\textbf{Brute-force} (heuristic)                                 \\
 \citet{cifuentes-fontanals_2022_Control} &      $\infty$ & Enumeration & \textbf{Symbolic} (model checking)              \\ \midrule
 \textbf{This paper}               &      Finite  & Enumeration & \textbf{Bilevel IP} (logic-based Benders)          \\ 
\bottomrule
\end{tabular}
\end{table}

Previous studies have mostly restricted the length of attractors to one. \cite{hayashida_2008_Algorithms} first developed an exponential-time search algorithm for \SACP{1}. Later, a few algorithms used \textit{integer programming} (IP) models by alternately solving two problems \citep{akutsu_2012_Integer,qiu_2014_Control}. For more efficient enumeration, computational algebra \citep{murrugarra_2016_Identification} and binary decision diagrams \citep{biane_2019_Causal} have been used. Recently, \cite{moon_2022_Bilevel} proposed a \emph{bilevel integer programming} (Bilevel IP) formulation of \SACP{1} and developed a scalable branch-and-bound algorithm.
On the other hand, research for $\Tmax > 1$ is quite limited. To the best of our knowledge, 
only \cite{cifuentes-fontanals_2022_Control} has addressed \InfiniteSACP{} using a model checking solver \citep{NuSMV2} with some heuristics developed by \cite{cifuentesfontanals_2020_Control}. However, the method often suffers from scalability issues. 

\Cref{tab:literature-review} summarizes the literature on \TmaxSACP{} and \InfiniteSACP{}. As the table shows, symbolic methods are the state-of-the-art algorithms, which is natural given the Boolean nature of the problem. Surprisingly, the recent Bilevel IP approach by \cite{moon_2022_Bilevel} showed superior performance over symbolic methods for \SACP{1}. This paper extends the bilevel IP framework of \cite{moon_2022_Bilevel} for $\Tmax=1$ to arbitrary finite $\Tmax$, filling the gap in the literature. We will later show that if $\Tmax$ is sufficiently large, the resulting algorithm can also be used for \InfiniteSACP{}.

\subsection{Bilevel Integer Programming and Logic-Based Benders Decomposition}

Bilevel IP \citep{kleinert_2021_Survey} involves two decision-makers, referred to as the \textit{leader} and the \textit{follower}, who make decisions sequentially as follows.
The leader first makes a decision anticipating the optimal reactions of the follower, and the follower then optimally solves their own optimization problem based on the leader's decision. The leader's
problem is referred to as the \textit{upper-level problem} (ULP) and the follower's problem is referred to as the
\textit{lower-level problem} (LLP). The main interest is to find an optimal solution to the ULP, which is formulated as a nested optimization problem. \emph{Coupling constraints} are constraints of the ULP that involve both the leader's and the follower's decision variables. They are important sources of computational challenges as they create a strong interdependence between the two problems.

\TmaxSACP{} is computationally challenging even with the restriction to attractor length one. Finding a single minimal control in \SACP{1} is $\Sigma^P_2$-hard \citep{akutsu_2012_Integer,moon_2022_Bilevel}, which implies that no general compact single-level reformulation is expected unless the polynomial hierarchy collapses \citep{stockmeyer1976Polynomialtime}. Hence, bilevel IP is not merely a convenient modeling choice here; it is the natural structure of the problem. 
For \TmaxSACP{}, the leader finds a minimal control, while the follower searches for an attractor that violates the phenotype under that control. Existing bilevel IP solvers such as \MibS{} \citep{tahernejad_2020_Branchandcut} can be used to solve \TmaxSACP{}. However, they rely on branch-and-bound and branch-and-cut methods \citep{moore_1990_Mixed,xu_2014_Exact,fischetti_2017_New,wang_2017_Watermelon,kleinert_2021_Survey}, which struggle to solve large instances of \TmaxSACP{}. We therefore develop a logic-based Benders decomposition specialized to this structure.

% \subsection{Logic-Based Benders Decomposition and Contributions}

% Logic-based Benders decomposition (LBBD) partitions a problem into a master problem and subproblem(s), then derives feasibility or optimality cuts from problem structure \citep{hooker_2003_Logicbased,hooker_2024_LogicBased,rahmaniani_2017_Benders}. This makes LBBD attractive for discrete models, but its use in bilevel IP is limited when the LLP contains integer variables.

The \emph{logic-based Benders decomposition} (LBBD) \citep{hooker_2003_Logicbased,hooker_2024_LogicBased} partitions an (usually single-level) optimization problem into a master problem and a subproblem. First, a \emph{master problem} decides one subset of variables subject to the constraints including those variables only. Given a candidate partial solution from the master problem, a \emph{subproblem} is solved to decide the remaining variables. If the candidate yields an infeasible or suboptimal subproblem, a corresponding Benders \emph{feasibility cut} or an \emph{optimality cut} is added to the master problem to remove the bad candidate. This process repeats until an optimal solution is found or all candidates are exhausted. 

The LBBD can be used for a wide variety of nonlinear and discrete optimization problems \citep{hooker_2024_LogicBased,rahmaniani_2017_Benders}. However, its application to bilevel IPs is limited. \emph{Benders binary cuts} \citep{denegre_2011_Interdiction} and \emph{Benders interdiction cuts} \citep{caprara_2016_Bilevel} are LBBD-style cuts available in the \MibS{} solver, but they require strict structural conditions. 
We address this gap with an \textit{infeasibility-based Benders decomposition} (IBBD) defined as follows.
\begin{definition}[Infeasibility-based Benders decomposition (IBBD)]\label{def:ibbd}
IBBD is a variant of LBBD for bilevel IPs, in which each subproblem is dedicated to finding a certain type of feasibility cut. The core idea is to decompose the subproblem by the type of coupling constraint violation as follows.
\begin{enumerate}[(i)]
\item \textbf{Master problem.}
	The coupling constraints and the lower-level optimality condition are relaxed from the ULP, leaving upper-level variables only.
\item \textbf{Aggregated subproblem.}
	Given a candidate, the LLP is solved with a lexicographic objective: the primary objective is the original LLP objective, and the secondary objective minimizes the violation of coupling constraints. % I presume lexicographic multiobjective is familiar to the OR community
	If the optimal violation is nonzero, the candidate is certified infeasible and a feasibility cut is derived.
\item \textbf{(Optional) Subproblem further decomposed by violation type.}
	If the violation type has multiple categories, the aggregated subproblem is further split into multiple subproblems by the type. This decomposition is valid if any candidate infeasible to ULP has at least one decomposed subproblem producing a feasibility cut. Subproblems are prioritized by their expected computational benefit. 
\end{enumerate}
\end{definition}
Here, we focus on applying IBBD to \TmaxSACP{}; \Cref{appendix:general-ibbd} discusses its generalization to other bilevel IPs with formal mathematical models.
In \TmaxSACP{}, a candidate control is infeasible to ULP if a \emph{forbidden attractor} is found under that candidate. An attractor is forbidden if at least one of its states violates the phenotype (i.e., $\phiFunction(\vbx)=0$ for some state $\vbx$ in the attractor). The aggregated subproblem is decomposed into smaller subproblems to detect a forbidden attractor of a specific length. We say $\Tmax$ is \emph{sufficiently large} when no control feasible to \TmaxSACP{} has a forbidden attractor longer than $\Tmax$ under that control. With a sufficiently large $\Tmax$, the \TmaxSACP{} is equivalent to \InfiniteSACP{}, and thus our algorithm can solve \InfiniteSACP{} as well.

In addition to the decomposed subproblems, we also solve an auxiliary subproblem that finds a \emph{fully forbidden trap space} where every state violates the phenotype. Since this includes one or more forbidden attractors, the resulting Benders cut is stronger under a certain condition, which is often satisfied in practice.

The main contributions of this paper are as follows.
\begin{enumerate}
	\item We extend the bilevel IP approach of \cite{moon_2022_Bilevel} from $\Tmax=1$ to arbitrary finite $\Tmax$ to capture more attractors that involve critical subsets of genes to control. With a sufficiently large $\Tmax$, the resulting algorithm can also be used for \InfiniteSACP{}, outperforming the state-of-the-art model checking approach by \cite{cifuentes-fontanals_2022_Control}.
	\item We propose a new type of logic-based Benders decomposition for the 
	synchronous attractor control problem, which decomposes the problem using the coupling constraint violation. We demonstrate how feasibility cuts can be strengthened by solving an auxiliary subproblem, and discuss how the entire framework can be generalized to bilevel IPs.
	\item We distribute the proposed algorithm as an open-source package \href{https://github.com/MSOLab/optboolnet}{\texttt{optboolnet}} (Software Heritage Identifier \href{https://archive.softwareheritage.org/swh:1:dir:194b2539a59ee22821b6b388b653376694045adc}{swh:1:dir:194b2539a59ee22821b6b388b653376694045adc}) via CoLoMoTo Docker \citep{naldi_2018_CoLoMoTo}, a Python platform for reproducible analyses of Boolean network models. Our package is expected to accelerate the discovery of new treatments with a reproducible computational pipeline.
\end{enumerate}

The remainder of this paper is organized as follows.
In \Cref{sec:model}, we present the bilevel IP
model for finding a single minimal control of \TmaxSACP{}. Based on this model, we develop the IBBD framework and present several proofs regarding the correctness and strength of Benders cuts in \Cref{sec:benders-algorithm}. Numerical experiments demonstrating the efficiency of the proposed algorithms are reported in \Cref{sec:computational-experiments}, followed by concluding remarks in \Cref{sec:conclusion}.

\section{Bilevel IP Model for the Synchronous Attractor Control Problem}\label{sec:model}

\subsection{Overview}
%
% \TmaxSACP{} is solved by repeatedly asking whether there exists a minimal control of a prescribed size $\TargetSize$ under which every attractor of length at most $\Tmax$ satisfies the phenotype. The leader chooses such a control, and the follower searches for a forbidden attractor under that control. We present the LLP first because it is the main modeling novelty: it encodes forbidden-attractor detection for all lengths up to $\Tmax$ in one formulation.

\begin{table}[tb]
	% \centering
	\TABLE{The notation for the bilevel integer programming model for \TmaxSACP \label{tab:parameters}}{
		\begin{tabular}{L{2.2cm}L{13.6cm}}
			\toprule
			\multicolumn{2}{l}{\textbf{Sets and parameters}}                                                                                                                                                                                                      \\ \midrule
			$\ControlSet$ / $\PrevControlSet$                             & The set of (all / previously found) minimal controls ($\PrevControlSet \subseteq \ControlSet$)                                                                                                                                                                 \\
			%%%
			$\GeneSet$ / $\CtrlGeneSet$ / $\UnctrlGeneSet$ & The set of (all / controllable / uncontrollable) genes  $(\CtrlGeneSet \subseteq \GeneSet, \UnctrlGeneSet := \GeneSet\setminus \CtrlGeneSet)$                                               \\
			$\phiFunction(\cdot)$                     & The phenotype ($\phiFunction: \B^{\numgenes} \rightarrow \B$)                                                                                                            \\
			$f_i(\cdot)$                              & The transition formula of gene $i$  ($f_i: \B^{\numgenes} \rightarrow \B$ for $i\in \GeneSet$)                                                                                \\
			$C^1_i$ / $C^0_i$                      & The  set of clauses in a CNF of ($f_i$ / $\neg f_i$) ($i \in \GeneSet$)                                                                                                                      \\
			$\ClauseSet^k$                           & $\bigcup_{i \in \GeneSet}C^k_i$ ($k \in \B$)                                                                                                                                                  \\
			$L^{+}_c$ / $L^{-}_c$                  & The set of genes appearing as (positive / negative) literals in clause $c$ ($c\in \ClauseSet^0 \cup \ClauseSet^1$)                                                                            \\
			%%%
			$\TargetSize$								   & The target size, a given number of genes to fix ($\TargetSize\in \mathbb{Z}^+ \cup\{0\}$)                                                                                                                                \\
			%%%
			$\Tmax$                            & The given maximum length of attractors ($\Tmax\in \mathbb{Z}^{+}$)                                                                                                                        \\
			%
			%%%
			%
			$\phigene$                                       & The auxiliary gene for the phenotype ($\phigene \in \UnctrlGeneSet$)                                                                                                                             \\
			\midrule
			\multicolumn{2}{l}{\textbf{Decision variables} [Relevant models$^{\dagger}$]}                                                                                                                                                      \\ \midrule
			$d^k_{j}$ \hfill $[\ULP, \SubspaceP]$
			& 1 if controllable gene $j$ is fixed to $k\in\{0,1\}$ and 0 otherwise ($j\in\CtrlGeneSet$) \\
			$\phiIndicator$ \hfill $[\ULP, \LLP]$
			& 1 if the phenotype is satisfied at all times and 0 otherwise \\
			$x_{i,t}$ \hfill $[\LLP]$
			& 1 if gene $i$ is activated at time $t$ and 0 otherwise
			($i\in\GeneSet,\ t\in \llbracket \Tmax\rrbracket$) \\
			$y_{c,t}$ \hfill $[\LLP]$
			& Truth value of clause $c$ in the given CNF formula at time $t$
			($c\in\ClauseSet^1,\ t\in \{0\}\cup\llbracket \Tmax\rrbracket$) \\
			$w_t$ \hfill $[\LLP]$
			& 1 if the selected attractor length is $t$ and 0 otherwise ($t\in \llbracket \Tmax\rrbracket$) \\
			$u^k_j$ \hfill $[\SubspaceP]$
			& 1 if controllable gene $j$ is fixed to be $k$ by a control and 0 otherwise ($j \in \CtrlGeneSet, k \in \B$) \\
			$h^k_i$ \hfill $[\SubspaceP]$
			& 1 if gene $i$ is fixed to be $k$ in the trap space and 0 otherwise ($i \in \GeneSet, k \in \B$) \\
			\bottomrule
		\end{tabular}}{
			$^{\dagger}$$\ULP$, $\LLP$, and $\SubspaceP$ denote variables of the upper-level problem, the lower-level problem, and the subspace separation, respectively. Bold and straight symbols are for vector notation.
			
			% A hat denotes a fixed value of a decision variable.
		}
		\end{table}
\Cref{tab:parameters} presents the notation used throughout the paper. We consider the gene set $\GeneSet$, where each gene $i\in\GeneSet$ has a Boolean activation level. Subsets $\CtrlGeneSet$ and $\UnctrlGeneSet$ denote the sets of \emph{controllable} and \emph{uncontrollable} genes, respectively. For decision variables, a plain symbol denotes a single variable, while a bold and straight symbol denotes a vector of all variables in the context.
We are given the maximum attractor length $\Tmax$. We use $\RangeOper{T}:=[1,\ldots,T]$ to denote a list of \emph{times} for a positive integer $T\leq \Tmax$. We denote by $x_{i,t}\in\B$ the activation level of gene $i$ at time $t$ in an attractor. A state vector at time $t$ is denoted as $\vbx_t$, which is a collection of $x_{i,t}$ for all $i\in \GeneSet$. If no control is applied, an attractor of length $T$ should satisfy $\vbx_{t}=f(\vbx_{t-1})$ for all $t\in\Trange$, considering time 0 equivalent to $T$.

We assume that both the transition formula of gene $i$ and its negation are given in Conjunctive Normal Form (CNF) as follows:
$f_i := \bigwedge_{c \in C^1_i} \left(\bigvee_{l \in L^{+}_{c}} x_l\ \vee \ \bigvee_{l \in L^{-}_{c}} \neg x_l\right),
\neg f_i := \bigwedge_{c \in C^0_i} \left(\bigvee_{l \in L^{+}_{c}} x_l\ \vee \ \bigvee_{l \in L^{-}_{c}} \neg x_l\right)$
where $C^1_i$ and $C^0_i$ are the sets of clauses and $L^{+}_c$ and $L^{-}_c$ are the sets of positive and negative literals in clause $c$, respectively. Conversion of a Boolean function to CNF can be done quickly in practice. We represent the truth value of clause $c$ at time $t$ by a variable $y_{c,t}\in\B$ (i.e., $y_{c,t} = (\bigvee_{l \in L^{+}_{c}} x_l\ \vee \ \bigvee_{l \in L^{-}_{c}} \neg x_l)$). The \emph{transition condition} for genes can be written as $x_{i,t} \leftrightarrow \bigwedge_{c \in C^1_i} y_{c,t-1}, \forall i\in \GeneSet, t\in \Tmaxrange$. If a control fixes gene $j\in \CtrlGeneSet$ to $k\in \B$, the transition condition is overridden by $x_{j,t} = k$ for all $t\in \Tmaxrange$.
We assume that there is an auxiliary, uncontrollable gene $\phigene$ such that its transition formula $f_\phigene$ equals the phenotype $\phiFunction$; any BN can be simply augmented to achieve this. Hence, the phenotype is satisfied at time $t$ if and only if $x_{\phigene,t}=1$.

A control is represented as a vector $\vbd\in \B^{2\numctrlgenes}$, where $d^{k}_j = 1$ indicates that the transition formula and thus the activation level of controllable gene $j$ is fixed to $k$ ($\forall j \in \CtrlGeneSet, k \in \B$). The \emph{size} of a control is the number of genes it fixes, the sum of values in $\vbd$. A control is minimal if there is no other feasible control $\hat{\vbd}$ that fixes a subset of genes (i.e., $\hat{d}^k_j \leq d^k_j$ for all $j \in \CtrlGeneSet$ and $k \in \B$).

We first propose ULP $\ULPModel$ and the LLP $\LLPModel$ to find a single minimal control of size $\TargetSize$ for \TmaxSACP{}. Here, $\TargetSize$ is a \emph{target size} of a control to be used in an outer enumeration loop of the final algorithm. $\PrevControlSet$ denotes the set of previously discovered minimal controls, to be used in a constraint ensuring the minimality of the control found by $\ULPModel$. We denote the set of all minimal controls as $\ControlSet$ ($\PrevControlSet \subseteq \ControlSet$).

Conceptually, the ULP seeks a control $\vbd$ with $\phiIndicator=1$, meaning that every represented attractor satisfies the phenotype. However, the LLP reacts to $\vbd$ by minimizing $\phiIndicator$ to construct a forbidden attractor if one exists. This bilevel structure ensures that all states of every attractor under $\vbd$ satisfy the phenotype.
Later, starting from $(\TargetSize,\PrevControlSet)=(0,\emptyset)$, we will repeatedly solve $\ULP(\TargetSize,\PrevControlSet)$ to enumerate all minimal controls of size $\TargetSize$ for \TmaxSACP{} and add them to $\PrevControlSet$; when no further minimal controls exist at size $\TargetSize$, we increment $\TargetSize$ and continue. This iterative procedure enumerates all minimal controls for \TmaxSACP{} by employing the same outer loop of \cite{moon_2022_Bilevel}.

\subsection{Mathematical model}

% ------------------------------
% Lower-level problem
% ------------------------------
\paragraph{\textbf{The lower-level problem} $\LLPModel$}
\begin{align}
\min
\quad & \phiIndicator
\label[objective]{eq:llp-obj} \\
\text{s.t.}\quad
& x_{j,t} \le 1-d^{0}_j,\quad
  x_{j,t} \ge d^{1}_j
&& \forall j\in \CtrlGeneSet,\ t\in \Tmaxrange
\label[constraint]{eq:llp-fix} \\
& x_{j,t} \le y_{c,t-1} + (d^{0}_j+d^{1}_j)
&& \forall j\in \CtrlGeneSet,\ c\in C_j^{1},\ t\in \Tmaxrange
\label[constraint]{eq:llp-con-1} \\
& \textstyle x_{j,t} \ge 1-\sum_{c\in C_j^{1}}(1-y_{c,t-1}) - (d^{0}_j+d^{1}_j)
&& \forall j\in \CtrlGeneSet,\ t\in \Tmaxrange
\label[constraint]{eq:llp-con-2} \\[3mm]
& x_{i,t} \le y_{c,t-1}
&& \forall i\in \UnctrlGeneSet,\ c\in C_i^{1},\ t\in \Tmaxrange
\label[constraint]{eq:llp-uncontrollable-1} \\
& \textstyle x_{i,t} \ge 1-\sum_{c\in C_i^{1}}(1-y_{c,t-1})
&& \forall i\in \UnctrlGeneSet,\ t\in \Tmaxrange
\label[constraint]{eq:llp-uncontrollable-2} \\[3mm]
& y_{c,t} \ge x_{l,t}
&& \forall c\in \ClauseSet^{1},\ l\in L_c^{+},\ t\in \Tmaxrange
\label[constraint]{eq:llp-literal-1} \\
& y_{c,t} \ge 1-x_{l,t}
&& \forall c\in \ClauseSet^{1},\ l\in L_c^{-},\ t\in \Tmaxrange
\label[constraint]{eq:llp-literal-2} \\
& \textstyle y_{c,t} \le \sum_{l\in L_c^{+}} x_{l,t} + \sum_{l\in L_c^{-}} (1-x_{l,t})
&& \forall c\in \ClauseSet^{1},\ t\in \Tmaxrange
\label[constraint]{eq:llp-literal-3} \\[3mm]
& \textstyle \sum_{t=1}^{\Tmax} w_t = 1
\label[constraint]{eq:llp-w-sum} \\
& -(1-w_t) \le y_{c,0} - y_{c,t} \le (1-w_t)
&& \forall c\in \ClauseSet^{1},\ t\in \Tmaxrange
\label[constraint]{eq:llp-wrap} \\[3mm]
& \phiIndicator \le x_{\phigene,t}
&& \forall t \in \Tmaxrange
\label[constraint]{eq:llp-ph-ub} \\
& \textstyle \phiIndicator \ge 1 - \sum_{t=1}^{\Tmax} (1 - x_{\phigene,t})
\label[constraint]{eq:llp-ph} \\[3mm]
& \mathrlap{
	\vbx \in \B^{\numgenes \Tmax}, \quad
	\vby \in \B^{\numclauses (\Tmax+1)}, \quad
	\vbw \in \B^{\Tmax}, \quad
	\phiIndicator \in \B.}
\label[constraint]{eq:llp-domain}
\end{align}

The symbol $\LLPModel$ denotes the LLP with the maximum attractor length $\Tmax$.
Given a control $\vbd$, \Cref{eq:llp-obj} minimizes $\phiIndicator$ to check whether there is a forbidden attractor of length up to $\Tmax$.

\paragraph{Transition condition.} \Crefrange{eq:llp-fix}{eq:llp-con-2} ensure the transition condition in an attractor for controllable genes: $d^k_j$ from the ULP appears in these LLP constraints.
\Cref{eq:llp-fix} directly fixes the activation level of a controlled gene to the designated value (i.e., $d^{k}_j \rightarrow (x_{j,t}=k), \forall j\in \CtrlGeneSet, k\in \B, t\in \Tmaxrange$).
\Crefrange{eq:llp-con-1}{eq:llp-con-2} represent the transition condition for a controllable gene that is not fixed, enforcing that it follows the original transition formula encoded by CNF
(i.e., $(\neg d^0_j \wedge \neg d^1_j) \rightarrow (x_{j,t} \leftrightarrow \bigwedge_{c \in C^1_j} y_{c,t-1}), \forall j\in \CtrlGeneSet, t\in \Tmaxrange$).
\Crefrange{eq:llp-uncontrollable-1}{eq:llp-uncontrollable-2} are the analogous transition constraints for uncontrollable genes, which are not affected by the given control.
\Crefrange{eq:llp-literal-1}{eq:llp-literal-3} define the CNF clause variables $y_{c,t}$ as the disjunction of the corresponding literals at the same time $t$ (i.e., $y_{c,t} \leftrightarrow (\bigvee_{l\in L_c^{+}} x_{l,t}\ \vee \ \bigvee_{l\in L_c^{-}} \neg x_{l,t}), \forall c\in \ClauseSet^{1},\ t\in \Tmaxrange$).

\paragraph{Attractor length selection.} \Crefrange{eq:llp-w-sum}{eq:llp-wrap} formulate the choice of an attractor length.
\Cref{eq:llp-w-sum} enforces only one length $T^*\in\Tmaxrange$ to be chosen.
If time $T^*$ is chosen, $y_{c,0}=y_{c,T^*}$ for every clause $c\in\ClauseSet^1$ by \Cref{eq:llp-wrap}; otherwise, the constraint becomes meaningless. This ensures the transition from time $T^*$ to time 1 by \Crefrange{eq:llp-fix}{eq:llp-literal-3}, constituting an attractor.
We note that, once \Cref{eq:llp-wrap} enforces $y_{c,0}=y_{c,T^*}$, the states at times from $(T^*+1)$ to $\Tmax$ simply repeat the attractor cycle; hence every $t\in\Tmaxrange$ corresponds to an attractor state regardless of the chosen length $T^*$.

\paragraph{Phenotype.} \Crefrange{eq:llp-ph-ub}{eq:llp-ph} formulate the phenotype condition of $\phiIndicator$.
\Crefrange{eq:llp-ph-ub}{eq:llp-ph} enforce that $\phiIndicator=1$ whenever the auxiliary gene for phenotype is always active (i.e., $\phiIndicator \leftrightarrow \bigwedge_{t=1}^{\Tmax} x_{\phigene,t}$).
Finally, \Cref{eq:llp-domain} specifies the binary domains of
all lower-level variables.

\paragraph{\textbf{The upper-level problem} $\ULPModel$}
{
% \small\allowdisplaybreaks
\begin{align}
\text{Find}\quad & (\vbd, \phiIndicator) \in \B^{2\numctrlgenes} \times \B \label[objective]{eq:ulp-objective} \\
\text{s.t.}\quad & d^{0}_j + d^{1}_j \leq 1
&& \forall j \in \CtrlGeneSet
\label[constraint]{eq:ulp-exclusivity} \\
& \textstyle \sum_{j \in \CtrlGeneSet} (d^{0}_j + d^{1}_j) = \TargetSize
\label[constraint]{eq:ulp-target-size} \\
& \textstyle \sum_{j\in \CtrlGeneSet}\sum_{k\in \B}
  \hat{d}^{k}_j\,(1-d^{k}_j) \ge 1
&& \forall \hat{\vbd} \in \PrevControlSet
\label[constraint]{eq:ulp-minimality} \\
& \phiIndicator = 1
\label[constraint]{eq:ulp-p-hard} \\
& \phiIndicator = \min\{\text{\Cref{eq:llp-obj}}: \text{\Crefrange{eq:llp-fix}{eq:llp-domain}}\}
\label[constraint]{eq:ulp-llp-opt}
% & \vbd \in \B^{2\numctrlgenes},\quad \phiIndicator \in \B
% \label[constraint]{eq:ulp-domain}
\end{align}
}

\Cref{eq:ulp-objective} shows that the ULP is formulated as a feasibility problem. 
\Cref{eq:ulp-exclusivity} enforces the exclusivity that each controllable gene can be controlled to be 0 or 1, but not both.
\Cref{eq:ulp-target-size} restricts the number of controlled genes to be
exactly the given target size $\TargetSize$.
\Cref{eq:ulp-minimality} eliminates any control that is a superset of a
previously discovered minimal control $\hat{\vbd}\in \PrevControlSet$. This ensures that any control found is minimal, given that we have exhaustively enumerated all minimal controls of size strictly less than $\TargetSize$.
\Cref{eq:ulp-p-hard} fixes $\phiIndicator=1$ to make $\vbd$ a feasible control.
\Cref{eq:ulp-llp-opt} enforces that $\phiIndicator$ equals the optimal objective value of the LLP.
\Cref{eq:ulp-p-hard,eq:ulp-llp-opt} are coupling constraints ensuring that the phenotype is satisfied for all states in every attractor of length at most $\Tmax$ under the given control $\vbd$.

As in standard bilevel IP, $\vbd$ can be a part of a feasible solution to $\ULPModel$ only when $\LLPModel$ is feasible. In our case, such infeasibility can happen if a chosen control induces no attractor within the given length bound $\Tmax$.

\section{An Infeasibility-Based Benders Decomposition for \TmaxSACP{}}\label{sec:benders-algorithm}
% We first specialize the IBBD framework to \TmaxSACP{}, then derive two Benders cuts---the exact attractor cut (\Cref{subsec:benders-attractor-cut}) and the optional trap space cut (\Cref{subsec:trap-space-cut})---and present the full algorithm (\Cref{subsec:benders-algorithm}).

%
We now present the IBBD framework in \Cref{def:ibbd} for \TmaxSACP{}. The master problem finds a candidate control $\vbd$ by relaxing the coupling constraints (\Cref{eq:ulp-p-hard,eq:ulp-llp-opt}).
Although the IBBD requires a lexicographic objective in the aggregated subproblem, minimizing the violation of \Cref{eq:ulp-p-hard} coincides with the original LLP objective \Cref{eq:llp-obj}. Hence, no secondary objective is needed.
The aggregated subproblem decomposes naturally by attractor length: if $\phiIndicator=0$, the LLP solution must contain a forbidden attractor of some length $T\in\Tmaxrange$. This yields $\Tmax$ decomposed subproblems, where the $T$-th subproblem fixes $w_T=1$ and removes variables indexed beyond $T$. If a subproblem finds a forbidden attractor, a Benders feasibility cut is added to the master problem to exclude the candidate $\vbd$. Since most forbidden attractors have small lengths in practice, we prioritize subproblems with smaller $T$, which also have fewer variables.

\paragraph{The Infeasibility-Based Benders Decomposition for \TmaxSACP{}}
\begin{align*}
	& \text{\textbf{\emph{Master problem}}} & \MPModel &\ := \text{ Find } \vbd\in \B^{2 \numctrlgenes}
	\text{ s.t. \Crefrange{eq:ulp-exclusivity}{eq:ulp-minimality}}, \text{ Benders cuts in } \CutSet
	\\
	& T\text{\textbf{\emph{-th subproblem}}} & \SubproblemModelAtT &\ := \LLPModelAtT{T} \text{ s.t. } (w_T=1)
\end{align*}
Any optimal solution of $\SubproblemModelAtT$ can be extended to a feasible solution of $\LLPModel$ by setting $w_t=0$ for $t>T$, fixing controlled genes to their prescribed values, and propagating uncontrolled genes through the transition formulas for the remaining time steps.
Therefore, we use the same domains $(\vbx, \vby, \vbw, p) \in \B^{\numgenes \Tmax} \times \B^{\numclauses (\Tmax+1)} \times \B^{\Tmax} \times \B$ for the decomposed subproblem variables.

\def \xmargin{0.7}
\def \xwidth{1.5}

\begin{figure}[t]
	\SingleSpacedXI

	\FIGURE{
		\centering
		\begin{tikzpicture}[
			problem/.style={draw, minimum width=\xwidth cm, minimum height=1cm, inner sep=0pt, align=center},
			flow/.style={-stealth, thick}
		]
			%
			% cut generation problems
			\node[problem, anchor=west] (subproblem-one) at (3.0,0.5) {$\SubproblemAt{\vbd}{1}$};
			\node[problem, anchor=west] (subproblem-two) at ($(subproblem-one.east)+(0.6cm,0)$) {$\SubproblemAt{\vbd}{2}$};
			\node at ($(subproblem-two.east)+(0.825cm,0)$) (subproblem-dots) {$\cdots$};
			\node[problem, anchor=west] (subproblem-last) at ($(subproblem-two.east)+(1.65cm,0)$) {$\SubproblemAt{\vbd}{\Tmax}$};
			\node[problem, anchor=east] (subspace) at ($(subproblem-one.west)+(-3.0cm,0)$) {$\SubspaceModel$};
			%
			% the master problem
			\coordinate (master-sw) at ($(subproblem-one.south west)+(-0.1cm,1.5cm)$);
			\coordinate (master-ne) at ($(subproblem-last.north east)+(0cm,1.5cm)$);
			\draw[] (master-sw) rectangle (master-ne);
			\node at ($(master-sw)!0.5!(master-ne)$) (master) {\textbf{The master problem} $\MPModel$};
			%
			% data flow
			\node[anchor=east] (candidate) at ($(master-sw)+(-0.8cm,-0.25cm)$) {Candidate $\vbd$};
			\node[anchor=west] (cut) at ($(master-ne)+(0.85cm,-1.25cm)$) {Benders cut};
			\coordinate (subproblem-entry) at ($(subproblem-one.west)+(-0.4cm,0)$);
			\coordinate (subproblem-exit) at ($(subproblem-last.east)+(0.45cm,0)$);
			\draw[flow] ($(master-sw)+(-0.3cm,0.5cm)$) .. controls +(-0.5cm,-0.3cm) and +(-0.5cm,0.3cm) .. (subproblem-entry);
			\draw[flow] (subproblem-exit) .. controls +(0.5cm,0.3cm) and +(0.5cm,-0.3cm) .. ($(master-ne)+(0.45cm,-0.5cm)$);
			%
			% bounding box for the subproblems
			\draw [gray] ($(subproblem-one.south west)+(-0.15cm,-0.15cm)$) rectangle ($(subproblem-last.north east)+(0.15cm,0.15cm)$);
			\draw [-latex,dashed,thick] ($(subspace.east)+(0.2cm,-0.2cm)$) -- ($(subproblem-entry)+(0,-0.2cm)$);
			\draw [-latex,dashed,thick] ($(master.west)+(-0.65cm,0.3cm)$) -- ($(subspace.north)+(0,1.3cm)$) -- (subspace.north);
			%
			% labels
			\draw [decorate, decoration = {calligraphic brace,mirror,amplitude=5pt},thick]
				($(subspace.south west)+(0,-0.3cm)$) -- ($(subspace.south east)+(0,-0.3cm)$)
				node[pos=0.5,below=10pt,black,align=center]{\textbf{The subspace separation}\\(The trap space cut; heuristic)};
			\draw [decorate, decoration = {calligraphic brace,mirror,amplitude=5pt},thick]
				($(subproblem-one.south west)+(0.05cm,-0.3cm)$) -- ($(subproblem-last.south east)+(0,-0.3cm)$)
				node[pos=0.5,below=10pt,black,align=center]{\textbf{Subproblems}\\(The attractor cut; exact)};
		\end{tikzpicture}}{Overview of the Benders cuts\label{fig:algorithm-overview}}{}
\end{figure}

We use two types of Benders feasibility cuts: attractor cuts and trap space cuts.
An \emph{attractor cut} (\attrcut{}) removes all candidates that make a certain forbidden attractor lower-level feasible. If the length of such a forbidden attractor is $T$, the resulting \attrcut{} is valid for $\ULPModel$ with any $\Tmax \geq T$.
On the other hand, a \emph{trap space cut} (\tscut{}) removes all candidates that have a fully forbidden trap space.
Since such trap space includes one or more forbidden attractors \citep{klarner_2015_Computing}, the resulting \tscut{} is also valid for $\ULPModel$ if $\Tmax$ is sufficiently large. Otherwise, \tscut{} should be interpreted as an efficient heuristic for finite $\Tmax$.

\Cref{fig:algorithm-overview} illustrates two IBBD iterations with or without \tscut{}s. In the default settings, only \attrcut{}s are generated by solving the subproblems in increasing order of $T$. If \tscut{}s are enabled, we first solve the \emph{subspace separation} $\SubspaceModel$, an IP that seeks a fully forbidden trap space under a candidate control $\vbd$. If such a trap space is found, a \tscut{} is generated to exclude the candidate $\vbd$; otherwise, we proceed to solve the decomposed subproblems as in the default settings.
\subsection{The Attractor Cut}\label{subsec:benders-attractor-cut}
In this section, we investigate how a (decomposed) subproblem characterizes the controls that induce a given attractor (\Cref{thm:attractor-characterization}) and use it to develop the \attrcut{} (\Cref{thm:attractor-cut}). We first define three important binary indicators in \Cref{def:binary-indicators}.

\begin{definition}[Binary indicators]\label{def:binary-indicators}
	Given $
	(\vbx,\vby)\in\B^{\numgenes \Tmax}\times
		\B^{\numclauses (\Tmax+1)}$ and controllable gene
	$j\in \CtrlGeneSet$, we define binary indicators as follows.
	%
	% $ \alphaj := \bigwedge_{t \in \Trange } (x_{j,t}   \leftrightarrow x_{j,1})$,
	% %
	% $\betaj := \bigwedge_{t \in \Trange } (x_{j,t}   \leftrightarrow  \bigwedge_{c\in C^1_j} y_{c,t-1})$,
	% %
	% and
	% $\kj := x_{j,1}$.
	% \begin{itemize}
	% 	\item $\alphaj := \bigwedge_{t \in \Tmaxrange } (x_{j,t}   \leftrightarrow x_{j,1})$, indicating if the activation levels are constant at all times.
	% 	\item $\kj :=  x_{j,1}$, indicating the activation level at time 1.
	% 	\item $\betaj := \bigwedge_{t \in \Tmaxrange } (x_{j,t}   \leftrightarrow  \bigwedge_{c\in C^1_j} y_{c,t-1})$, indicating if the update follows the transition formula at all times.
	% \end{itemize}
	\begin{align*}
		\alphaj & := \bigwedge_{t \in \Tmaxrange } (x_{j,t}   \leftrightarrow x_{j,1}) && \text{(the activation levels are constant at all times)} \\
		\betaj & := \bigwedge_{t \in \Tmaxrange } (x_{j,t}   \leftrightarrow  \bigwedge_{c\in C^1_j} y_{c,t-1}) && \text{(the update follows the transition formula at all times)} \\
		\kj & :=  x_{j,1} && \text{(the activation level at time 1)}
	\end{align*}
\end{definition}
%
% Here, $\alphaj$ indicates whether gene $j$ has the same value at all times. 
% When $\alphaj=1$, $\kj$ denotes the fixed value. Otherwise, $\kj$ has no meaning. $\betaj$ indicates whether gene $j$ is  updated following the transition formula for all times. 

If gene $j$ is controlled, then $\alphaj =1$. Otherwise, we have $\betaj =1$ to satisfy the transition formula, which implies that $(1-\alphaj)(1-\betaj)=0$ for all $j\in\CtrlGeneSet$ if $\vbx$ forms an attractor under some control. This leads to the following characterization of feasible controls for a synchronous attractor.
\begin{theorem}[Characterization of feasible controls for a synchronous attractor]\label{thm:attractor-characterization}
	Let $(\hat{\vbx},\hat{\vby},\hat{\vbw},\hat{\phiIndicator})$ be a feasible solution of $\SubproblemAt{\hat{\vbd}}{T}$ for some $T\in\Tmaxrange$ and feasible control $\hat{\vbd}$.
	Let $\alphaj$, $\betaj$, and $\kj$ be the binary indicators defined in \Cref{def:binary-indicators} for each $j\in \CtrlGeneSet$ based on $(\hat{\vbx},\hat{\vby})$.
	% , so $(\vbx,\vby)$ satisfies \Crefrange{eq:llp-uncontrollable-1}{eq:llp-uncontrollable-2} and \Crefrange{eq:llp-literal-1}{eq:llp-literal-3}.
	% 
	For all $j\in \CtrlGeneSet$,
    define  $g_j(\hat{\vbx},\vbd)$ as an affine function of $\vbd$, given by
	\begin{subequations}
		\begin{align}
			g_j(\hat{\vbx},\vbd) &= d^{(1-\kj)}_j
			&& \text{if } \phantom{\neg}\alphaj \wedge \betaj
			\label{eq:attractor-characterization-1} \\
			g_j(\hat{\vbx},\vbd) &= 1 - d^{\kj}_j
			&& \text{if } \phantom{\neg}\alphaj \wedge \neg \betaj
			\label{eq:attractor-characterization-2} \\
			g_j(\hat{\vbx},\vbd) &= d^0_j + d^1_j
			&& \text{if } \neg \alphaj \wedge \betaj
			\label{eq:attractor-characterization-3}
		\end{align}
	\end{subequations}
	Then, for a control $\vbd$ satisfying \Cref{eq:ulp-exclusivity}, $\hat{\vbx}$ is an attractor under control $\vbd$ if and only if
	$g_j(\hat{\vbx},\vbd) = 0$ for all $j \in \CtrlGeneSet$.
\end{theorem}
The formal proof of \Cref{thm:attractor-characterization} is presented in \Cref{subsec:proof-attractor-characterization}. \Cref{thm:attractor-characterization} shows that $\alphaj$ and $\betaj$ are factors that determine the feasibility of a control with the following rules. If a controllable gene has the same value at all times ($\alphaj=1$), the required condition may depend on whether the gene's state is consistently updated following the transition formula at all times (i.e., whether $\betaj=1$). If it follows ($\betaj=1$),
it is sufficient that the control does not fix the gene to its negated value due to \Cref{eq:attractor-characterization-1}. If not ($\betaj=0$), however, a control must
fix the gene to have the current value due to \Cref{eq:attractor-characterization-2}.
On the other hand, when a controllable gene does not have a constant value over time,
no control should be applied due to \Cref{eq:attractor-characterization-3}.
The remaining case with $\neg \alphaj \wedge \neg \betaj$
is never activated for any attractor solution and can therefore be disregarded.

\begin{theorem}[The attractor cut]\label{thm:attractor-cut}
	Given a candidate control $\hat{\vbd}$ satisfying \Cref{eq:ulp-exclusivity}, suppose $(\hat{\vbx},\hat{\vby},\hat{\vbw},0)$ is an optimal solution to $\SubproblemAt{\hat{\vbd}}{T}$ (i.e., $\hat{\vbx}$ is a forbidden attractor).
	Then, 
	% $\sum_{j \in \CtrlGeneSet} g_j(\hat{\vbx},\hat{\vbd}) = \sum_{j \in \CtrlGeneSet} \left((1-\alphaj) d^{\kj}_j + \betaj  d^{(1-\kj)}_j  +  (1-\betaj) (1-d^{\kj}_j) \right) \geq 1$
	% 
	\begin{align}
		\sum_{j \in \CtrlGeneSet} g_j(\hat{\vbx},\vbd) = &  \sum_{j \in \CtrlGeneSet} \left((1-\alphaj) d^{\kj}_j + \betaj  d^{(1-\kj)}_j  +  (1-\betaj) (1-d^{\kj}_j) \right) \geq 1\label{eq:attractor-cut}
	\end{align}
	is a
	valid constraint for all feasible solutions to $\ULPModel$
	for any $\TargetSize \in \mathbb{Z}^+ \cup\{0\}$ and $\PrevControlSet\subseteq \ControlSet$. When added to the master problem, this constraint removes $\hat{\vbd}$ from the solution space as well as all controls $\vbd$ that induce
    $\hat{\vbx}$ as an optimal solution of $\SubproblemModelAtT$.
\end{theorem}
The formal proof of \Cref{thm:attractor-cut} is presented in \Cref{subsec:proof-attractor-cut} based on the fact that $\hat{\vbx}$ must not be an attractor under any control. 
The \attrcut{} may involve the term $d^{(1-\kj)}_j$
or $1-d^{\kj}_j$, based on the value of $\betaj$. Due to \Cref{eq:ulp-exclusivity}, we have $d^{(1-\kj)}_j \leq (1-d^{\kj}_j)$. This implies that the
\attrcut{} gets stronger when a controllable gene is updated by the transition
formula for all times. Another remark is that if the actual length of the attractor is 1,
the \attrcut{} is equivalent to the Benders feasibility cut
proposed by \cite{moon_2022_Bilevel} to solve \SACP{1}. Hence, \Cref{eq:attractor-cut} serves as a proper extension for attractors of greater length.

%%%
%%%
\subsection{The Trap Space Cut: Strengthened Cut for Large $\Tmax$}\label{subsec:trap-space-cut}
%%%
%%%
An \attrcut{} eliminates candidate controls based on a single forbidden attractor. However, if multiple forbidden attractors share a common fully forbidden trap space, the \tscut{} can aggregate them into a single stronger cut under certain conditions. 
We first show that a fully forbidden trap space may remain fully forbidden under some controls (\Cref{thm:control-preserving-trap-spaces}). Then, with a sufficiently large $\Tmax$, the resulting \tscut{} is also valid for $\ULPModel$ (\Cref{thm:trap-space-cut}).
In this section, we denote a trap space as
$\vbh\in\B^{2\numgenes}$ with $h^k_i=1$ if and only if gene $i$
is fixed to the value $k$ in the trap space for $i\in \GeneSet$ and $k\in
	\B$.
By definition, $h^0_i+h^1_i\leq 1$ holds.
We say a state $\vbx$ \textit{belongs to} a trap space
$\vbh$ if the fixed values are consistent (i.e., $\forall i \in
	\GeneSet, k\in \B, (h^k_i=1) \rightarrow (x_i = k)$). We denote the control
inducing a trap space as $\vbu \in \B^{2 \numctrlgenes}$ to avoid
confusion with a candidate control $\vbd$.

\begin{theorem}[Trap spaces remaining under a control]\label{thm:control-preserving-trap-spaces}
	Suppose $\vbh\in \B^{2 \numgenes}$ is a given trap space under control $\vbu\in \B^{2  \numctrlgenes}$ that satisfies
	\Cref{eq:ulp-exclusivity}.
	Let $\vbd \in \B^{2 \numctrlgenes}$ be a candidate control that
	satisfies \Cref{eq:ulp-exclusivity} and $u^k_j \leq d^k_j$ for all $j\in
		\CtrlGeneSet, k \in \B$. Let $f'$ be the resulting Boolean network under control
	$\vbd$, defined by $f'_i(\vbx) = k$ if $i \in \CtrlGeneSet$ and $d^k_i = 1$ for some $k\in\B$; otherwise $f'_i(\vbx) = f_i(\vbx)$.
	% \begin{align*}
	% 	f'_i(\vbx) =
	% 	\begin{dcases}
	% 		k               & \text{if}~ d^k_i = 1 \text{ for }   i\in \CtrlGeneSet, k\in\B \\
	% 		f_i(\vbx) & \text{otherwise}
	% 	\end{dcases}.
	% \end{align*}
	%
	It holds that $\vbh$ is a trap space of Boolean network $f'$ if 
	% $(h^k_j=1) \rightarrow (d^{(1-k)}_j = 0)$ for all $j \in \CtrlGeneSet$ and $k \in \B$.
	\begin{align}
		(h^k_j=1) \rightarrow (d^{(1-k)}_j = 0)
		 & \quad \forall j \in \CtrlGeneSet, k \in \B.\label{eq:trap-space-condition}
	\end{align}
\end{theorem}
\begin{proof}{Proof of \Cref{thm:control-preserving-trap-spaces}}
	Assume, for contradiction, that the condition $(h^k_j=1) \rightarrow (d^{(1-k)}_j = 0)$ holds but $\vbh$ is not a trap space under control $\vbd$. Let $\mathcal{F} (\subseteq \CtrlGeneSet)$ denote the set of all genes fixed by
	$\vbu$.
	Then, there is a state $\vbx\in \B^{\numgenes}$, a gene $i \in \GeneSet$, and $k\in \B$ such that $\vbx$ belongs to $\vbh$ and $h^k_i=1$ but $f'_i(\vbx) = 1-k$. 
	If $i \in \UnctrlGeneSet$, $f'_i(\vbx) = f_i(\vbx)$ since $i$ is uncontrollable. As $\vbh$ is a trap space under $\vbu$ and $h^k_i=1$, we have $f_i(\vbx) = k (\neq 1-k)$, a contradiction. 
	If $i \in \CtrlGeneSet \cap \mathcal{F}$, this also cannot happen since $u^k_i \leq d^k_i$ implies that $f'_i(\vbx) = k$. 
	For the remaining case where $i \in \CtrlGeneSet \cap \mathcal{F}^\text{c}$, we consider two sub-cases. 
	If $d^k_i = 1$, then $f'_i(\vbx) = k (\neq 1-k)$, a contradiction. 
	If $d^k_i = 0$, then since $(h^k_i=1) \rightarrow (d^{(1-k)}_i = 0)$ implies $d^{(1-k)}_i = 0$, we have $d^0_i = d^1_i = 0$ and thus $f'_i(\vbx) = f_i(\vbx)$. Since $i \notin \mathcal{F}$, the BN under $\vbu$ also leaves gene $i$ uncontrolled, so $f_i(\vbx) = k$ follows from $\vbh$ being a trap space under $\vbu$ and $h^k_i=1$. This gives $f'_i(\vbx) = k (\neq 1-k)$, again a contradiction. Hence, every possible choice of $i$ contradicts the hypothesis, and thus $\vbh$ is a trap space under control $\vbd$.
	\hfill\Halmos
\end{proof}

As illustrated by the example in \Cref{fig:example}, trap space $(**1)$
remains a trap space under the control that fixes $x_2=1$. This is because
the control only affects the unfixed value of gene 2. On the other hand, the trap space $(000)$ is no longer a trap space under that control because the fixed value 0 is forced to its negated value. These observations align with \Cref{thm:control-preserving-trap-spaces}.

\begin{theorem}[The trap space cut]\label{thm:trap-space-cut}
	Let $\vbh \in \B^{2\numgenes}$ be a fully forbidden trap space under control $\vbu \in \B^{2\numctrlgenes}$ such that for all $\vbx$ belonging to $\vbh$, $x_{\phigene}=0$ holds.
	If $\Tmax$ is sufficiently large,
	\begin{align}
		\sum_{j \in \CtrlGeneSet} \sum_{k \in \B} (u^k_j(1- d^{k}_j) + (1-u^k_j)h^k_j d^{(1-k)}_j) \geq 1
		 & \label{eq:trap-space-cut}
	\end{align}
	is a valid constraint for all feasible solutions to $\ULPModel$
	for any $\TargetSize \in \mathbb{Z}^+ \cup\{0\}$ and $\PrevControlSet\subseteq \ControlSet$.
\end{theorem}
The formal proof of \Cref{thm:trap-space-cut} is presented in  \Cref{subsec:proof-trap-space-cut}. A feasible control must either (i) release some of the fixed values in
$\vbu$ or (ii) convert some of the fixed values in $\vbh$ to their negated values; otherwise, $\vbh$ remains fully forbidden by \Cref{thm:control-preserving-trap-spaces}. With a sufficiently large $\Tmax$, at least one forbidden attractor is included in the given trap space $\vbh$. The first term of \Cref{eq:trap-space-cut} corresponds to case (i), and the second term corresponds to case (ii).

Based on \Cref{thm:trap-space-cut}, we formulate the \emph{subspace separation} $\SubspaceModel$, an IP that seeks a control $\vbu$ and a corresponding fully forbidden trap space $\vbh$ from which a \tscut{} can be derived. This model is solved for the current candidate control $\vbd$, which appears in \Crefrange{eq:separation-1-implied-coeff-1}{eq:separation-1-implied-coeff-2}. Variable $u^k_j$ is 1 if controllable gene $j$ is fixed to be $k$ by a control, and 0 otherwise ($j \in \CtrlGeneSet, k \in \B$); $h^k_i$ is 1 if gene $i$ is fixed to be $k$ in the trap space, and 0 otherwise ($i \in \GeneSet, k \in \B$).

\paragraph*{\textbf{The subspace separation} $\SubspaceModel$}
%

% \small
\begin{align} 
	\text{min }
	 & \textstyle \sum_{j \in \CtrlGeneSet} \sum_{k \in \B} h^{k}_j                                                                                                                  & \label[objective]{eq:separation-IP-obj}                                                                                                 \\
	\text{ s.t. }
	 & u^{0}_j + u^{1}_j                                         \leq 1                                                                                                          &                              & \forall j \in \CtrlGeneSet\label[constraint]{eq:separation-IP-exclusivity-1}                    \\
	 & h^{0}_i + h^{1}_i                                         \leq 1                                                                                                          &                              & \forall i \in \GeneSet\label[constraint]{eq:separation-IP-exclusivity-2}                    \\
	 & u^k_j                                                     \leq  h^{k}_j                                                                                                   &                              & \forall j \in \CtrlGeneSet, k \in \B\label[constraint]{eq:separation-IP-control}        \\
	 & \textstyle h^k_i                                                     \leq \sum_{i' \in L^+_c} h^1_{i'} + \sum_{i' \in L^-_c} h^0_{i'} & &  \forall i \in \UnctrlGeneSet, k\in \B, c\in C^k_i\label[constraint]{eq:separation-IP-implied-fix-1} \\
	 & \textstyle h^k_j - u^k_j \leq \sum_{i' \in L^+_c} h^1_{i'} + \sum_{i' \in L^-_c} h^0_{i'} & &   \forall j \in \CtrlGeneSet, k\in \B, c\in C^k_j\label[constraint]{eq:separation-IP-implied-fix-2}          \\
	 & h^0_{\phigene}                                                     = 1\label[constraint]{eq:separation-IP-forbidden-trap-space}                                                                                                                                                                                               \\
	 & u^k_j + (1-h^k_j)                                         \geq d^{(1-k)}_j                                                                               &                              & \forall j \in \CtrlGeneSet, k \in \B\label[constraint]{eq:separation-1-implied-coeff-1} \\
	 & u^{k}_j                                                   \leq d^k_j                                                                                                &                              & \forall j \in \CtrlGeneSet, k \in \B\label[constraint]{eq:separation-1-implied-coeff-2} \\
	 & \vbu, \vbh                                    \in \B^{2\numctrlgenes} \times \B^{2\numgenes}\label[constraint]{eq:separation-IP-domain}
\end{align}

\Cref{eq:separation-IP-obj} is equivalent to minimizing the number of nonzero literals in the \tscut{}
$(=\sum_{j \in \CtrlGeneSet}\sum_{k \in \B}(u^k_j + (1 - u^k_j) h^k_j))$
% $ \sum_{j \in \CtrlGeneSet}
% 	\sum_{k \in \B}
% 	(u^k_j + (1 - u^k_j) h^k_j)   = \sum_{j \in \CtrlGeneSet} \sum_{k \in \B} (u^k_j(1 - h^k_j) + h^k_j) = \sum_{j \in \CtrlGeneSet} \sum_{k \in \B} h^k_j$
%
since $(u^k_j = 1) \rightarrow (h^k_j=1)$ by \Cref{eq:separation-IP-control}.
\Cref{eq:separation-IP-exclusivity-1,eq:separation-IP-exclusivity-2} ensure that, for each gene, the control and trap space variables cannot simultaneously take the values 0 and 1.
\Cref{eq:separation-IP-control} fixes the corresponding variable of the selected
trap space if it is controlled. 
Since $f_i$ and $\neg f_i$ are CNFs, they are tautologies (i.e., always true) if and only if every clause is also a tautology. 
\Cref{eq:separation-IP-implied-fix-1} ensures this condition for $f_i$ and $\neg f_i$ if uncontrollable gene $i$ is fixed as 1 and 0 in the selected trap space, respectively.
\Cref{eq:separation-IP-implied-fix-2} serves a similar role for the
controllable genes,
but they are active only when no control is applied.
\Cref{eq:separation-IP-forbidden-trap-space} restricts the selected
trap space to be a fully forbidden trap space.
\Crefrange{eq:separation-1-implied-coeff-1}{eq:separation-1-implied-coeff-2}
provide the conditions under which the derived \tscut{} will remove the candidate $\vbd$. Evaluating \Cref{eq:trap-space-cut} at $\vbd$ yields a left-hand side of 0 due to the following facts.
Given $j\in \CtrlGeneSet, k\in \B$, if
$d^{(1-k)}_j = 1$, then by exclusivity $d^k_j = 0$, and the corresponding term becomes $u^k_j + (1-u^k_j)
	h^k_j$. \Cref{eq:separation-1-implied-coeff-2} forces $u^k_j \leq d^k_j = 0$, reducing the term to $h^k_j$, which is then forced to be 0 by
\Cref{eq:separation-1-implied-coeff-1}. On the other hand, if
$d^k_j = 0$, the corresponding term becomes $u^k_j$, which is
implied to be 0 by \Cref{eq:separation-1-implied-coeff-2}.
Therefore, all terms in \Cref{eq:trap-space-cut} are 0 and $\vbd$
violates the constraint.
Notably, the decision version of the subspace separation is equivalent to deciding
whether $\vbh$ is a trap space if $h^0_{\phigene}=1$ and $\vbd=\vbzero$. This problem is known to be coNP-complete \citep{moon_2024_Computational}.

\begin{theorem}\label{thm:separation-1}
	If $\Tmax$ is sufficiently large and $(\vbu, \vbh) \in \B^{2\numctrlgenes} \times \B^{2\numgenes}$ is a solution to the subspace separation IP~$\SubspaceModel$, then the trap space cut \Cref{eq:trap-space-cut}
	is a valid constraint for all feasible solutions to $\ULPModel$
	for any $\TargetSize \in \mathbb{Z}^+ \cup\{0\}$ and $\PrevControlSet\subseteq \ControlSet$.
	Furthermore,
   this constraint excludes the candidate $\vbd$ from the feasible region.
\end{theorem}
\begin{proof}{Proof of \Cref{thm:separation-1}}
	Since $(\vbu, \vbh)$ satisfies \Crefrange{eq:separation-IP-exclusivity-1}{eq:separation-IP-forbidden-trap-space}, $\vbh$ is a fully forbidden trap space under control $\vbu$. If $\Tmax$ is sufficiently large, the validity of the \tscut{} in \Cref{eq:trap-space-cut} is a direct consequence of \Cref{thm:trap-space-cut}. Candidate $\vbd$ is cut off by \Crefrange{eq:separation-1-implied-coeff-1}{eq:separation-1-implied-coeff-2}.
	\hfill\Halmos
\end{proof}

The \attrcut{} is the exact baseline because it is derived from a specific forbidden attractor returned by a subproblem. The \tscut{} becomes attractive when a fully forbidden trap space includes many such forbidden attractors and a single aggregated cut can replace several attractor cuts. \Cref{thm:stronger-benders-cut} provides a condition under which the \tscut{} is provably stronger than any \attrcut{} derived from a forbidden attractor within the fully forbidden trap space. In other words, the \tscut{} successfully aggregates various \attrcut{}s satisfying the condition into a stronger cut.

\begin{theorem}[The strength of a trap space cut]\label{thm:stronger-benders-cut}
	Suppose $(\vbu, \vbh) \in \B^{2\numctrlgenes} \times \B^{2\numgenes}$ is a solution to the subspace separation~$\SubspaceModel$.
	Let $(\vbx, \vby, \vbw, 0)$ be an optimal solution to $\SubproblemAt{\vbu}{T}$ such that $\vbx$ belongs to the fully forbidden trap space $\vbh$ under control $\vbu$.
	If the implication $(u^0_j \vee u^1_j)\rightarrow \neg \betaj$ holds for all $j\in \CtrlGeneSet$, the trap space cut \Cref{eq:trap-space-cut} implies the attractor cut \Cref{eq:attractor-cut} derived with $\vbx$.
\end{theorem}

The formal proof of \Cref{thm:stronger-benders-cut} is presented in \Cref{subsec:proof-stronger-benders-cut}.
As discussed above, the \attrcut{} gets stronger when a controllable gene is consistently updated by the transition
formula for all times.
However, the \tscut{} cannot capture this property, and thus
\Cref{thm:stronger-benders-cut} may not hold if the condition $(u^0_j \vee
	u^1_j)\rightarrow \neg \betaj$ is not true for some
$j\in\CtrlGeneSet$; we present such a counterexample with two controllable genes in \Cref{subsec:counterexample}. Nevertheless, the \tscut{} is much stronger than the \attrcut{} in practice. We provide a numerical comparison of the two Benders cuts in \Cref{sec:computational-experiments}.

\subsection{Repeated IBBD Solving to Enumerate All Minimal Controls}\label{subsec:benders-algorithm}

\begin{algorithm}[t]
	% \TableSpaced
	\small
	\caption{The enumerative IBBD algorithm for \TmaxSACP{}}\label{alg:benders-algorithm}
	\BlankLine{}
	\SetKwInOut{Data}{Input}
	\SetKwInOut{Result}{Output}
	\SetKw{Continue}{continue}
	\SetKw{Break}{break}
	\SetKw{And}{and}
	\SetKw{In}{in}
	\SetKw{Or}{or}
	\SetNoFillComment
	\KwIn{Boolean network $f$, phenotype $\phiFunction$, the maximum attractor length $\Tmax$, and option $\UseTScut\in\{0,1\}$}
	\KwOut{The set of minimal controls $\PrevControlSet$}
	\Begin{
		$(\PrevControlSet, \CutSet) \leftarrow (\emptyset, \emptyset)$\;\label{alg-line:initialization}
		\For{$\TargetSize$ \In $[0, \ldots, \numctrlgenes]$\label{alg:loop-target-size-start}}{
			\While{$\MPModel$ has a feasible solution $\vbd$}{
				\If{$\UseTScut=1$ and $\SubspaceModel$ has a feasible solution $(\vbu,\vbh)$\label{alg-line:subspace-separation-1}}{
					$\CutSet \leftarrow \CutSet \cup \{\text{The trap space cut } \Cref{eq:trap-space-cut} \}$\label{alg:trap-space-cut} \And
					\Continue\;\label{alg-line:subspace-separation-2}
				}
				Find the smallest $T \in \Tmaxrange$ such that $\SubproblemModelAtT$ has an optimal solution with $\phiIndicator = 0$\;\label{alg-line:subproblem}
				\lIf{(Such $T$ exists with optimal solution $(\vbx, \vby, \vbw, 0)$)}{
					$\CutSet \leftarrow \CutSet \cup \{\text{The attractor cut } \Cref{eq:attractor-cut} \}$\label{alg:attractor-cut}}
				\lElseIf{($\SubproblemModelAtT$ infeasible for all $T \in \Tmaxrange$)}{
					$\CutSet \leftarrow \CutSet \cup \{
						\text{The no-good cut } \Cref{eq:no-good-cut}
						\}$\label{alg-line:llp-end}}
				\lElse{
					$\PrevControlSet \leftarrow \PrevControlSet \cup \{\vbd\}$}
					}
		}\label{alg:loop-target-size-end}
		\Return{$\PrevControlSet$}
	}
\end{algorithm}

\Cref{alg:benders-algorithm} enumerates all minimal controls in nondecreasing order of size through \Cref{alg:loop-target-size-start}--\ref{alg:loop-target-size-end}. 
When $\UseTScut=1$, the algorithm first tries the subspace separation $\SubspaceModel$ to obtain a usually stronger \tscut{} (\Cref{alg:trap-space-cut}). If no such \tscut{} is found, or if $\UseTScut=0$, it searches the subproblems $\SubproblemModelAtT$ from small $T$'s and checks whether their optimal value $\phiIndicator$ equals 0. If the optimal solution of some subproblem has $\phiIndicator=0$, the corresponding forbidden attractor generates an \attrcut{} (\Cref{alg:attractor-cut}). If every subproblem is infeasible, then no attractor of length at most $\Tmax$ exists under $\vbd$, so we add the no-good cut, which excludes only the specific candidate $\hat{\vbd}$:
\begin{align}
	\sum_{j\in \CtrlGeneSet}\sum_{k\in \B} \left(\hat{d}^k_j(1-d^k_j) + (1-\hat{d}^k_j)d^k_j\right) \geq 1. \label{eq:no-good-cut}
\end{align}
Otherwise, $\vbd$ is feasible for $\ULPModel$ and is added to $\PrevControlSet$. This algorithm has the following properties.
\begin{itemize}
	\item \textbf{Correctness}: When $\UseTScut=0$, the algorithm exactly solves \TmaxSACP{} via \attrcut{}s and no-good cuts. When $\UseTScut=1$, the same baseline remains and \tscut{}s act as an optional strengthening with sufficiently large $\Tmax$.
	\item \textbf{Termination}: The algorithm terminates in a finite number of runs of the master problem since (i) the number of candidates is finite, and (ii) one or more candidates are removed in each run either by Benders cuts or \Cref{eq:ulp-minimality} generated by a new solution.
\end{itemize}

\section{Computational Experiments}\label{sec:computational-experiments}

\subsection{Experiment Setting}\label{subsec:experiment-settings}

We compare the following methods.

\begin{itemize}
   \item \MibS{}: Exhaustively solving $\ULPModel$ with a bilevel IP solver \MibS{} \citep{tahernejad_2020_Branchandcut} (ver.~1.2) and adding \Cref{eq:ulp-minimality} for minimality; this replaces \Cref{alg:loop-target-size-start}--\ref{alg:loop-target-size-end} in \Cref{alg:benders-algorithm}.

   \item \PBN{}: The model-checking approach of \cite{cifuentes-fontanals_2022_Control}, implemented using the Python package \texttt{pyboolnet}. \PBN{} targets \InfiniteSACP{} and does not rely on a prescribed $\Tmax$.

   \item \BEN{}/\SEP{}: Enumerative IBBD algorithms based on \Cref{alg:benders-algorithm}. \BEN{} sets $\UseTScut=0$ and uses only \attrcut{}s, whereas \SEP{} sets $\UseTScut=1$ and additionally uses \tscut{}s. Both accept two subproblem options for \Cref{alg-line:subproblem}:
   \begin{itemize}
      \item (DEC): Solve the decomposed subproblems $\SubproblemModelAtT$ for $T=1,\ldots,\Tmax$.
      \item (AGG): Solve $\LLPModel$, and add an \attrcut{} if its optimal solution has $\phiIndicator=0$. We also label \MibS{} as AGG since it uses the same model.
   \end{itemize}
\end{itemize}

Our experiments address the following questions: 
\begin{enumerate}[(i)]
\item How do the proposed IBBD algorithms (\BEN{} and \SEP{}) perform compared to \MibS{} and \PBN{}?
\item Does the use of the trap space cut lead to better performance (\SEP{} vs. \BEN{})?
\item Do the decomposed subproblems (DEC) outperform the aggregated subproblems (AGG)?
\item How does the choice of $\Tmax$ affect computational performance and feasibility for \InfiniteSACP{}?
\end{enumerate}

All algorithms enumerate minimal controls in nondecreasing order of their size. IP models are solved with \texttt{Gurobi} (ver. 12.0.0). Since the \tscut{} was originally designed for \InfiniteSACP{}, \BEN{} serves as our exact finite-$\Tmax$ method and \SEP{} as a heuristic finite-$\Tmax$ variant.

We use the eleven BN instances from \cite{moon_2022_Bilevel}, which are biologically relevant disease models with clear phenotype interpretations.
We grouped them by their sizes into small (13--20 genes), medium (28--35 genes), and large (53--75 genes) instances; other statistics are provided in \Cref{appendix:comp-results}. 
Our main performance measure is the number of minimal controls found within 600 seconds, reflecting time-critical medical applications where the discovered controls are passed to further experimental validation. Bold entries indicate that the algorithm certified completeness for all minimal controls of size $\leq 7$.
All experiments were run on a Windows 11 PC with an Intel Core i7-14700KF CPU and 64 GB RAM. Computation time is reported in wall-clock seconds.

\subsection{Results}\label{subsec:results}

\newcommand{\expna}{-}
\newcommand{\ctrlnote}{\multicolumn{3}{c}{\textbf{Algorithm}}      & \multicolumn{11}{c}{\textbf{Instances}} }
\newcommand{\graycmidrules}[1]{%
	\noalign{\vskip-\aboverulesep}%
	\arrayrulecolor{lightgray}#1\arrayrulecolor{black}%
	\noalign{\vskip-\belowrulesep}%
}

\newcommand{\numcontrolnote}{
	\multicolumn{1}{r}{\small (\textbf{bold}):} & \multicolumn{12}{l}{\small The algorithm proved that all minimal controls of size $\leq 7$ were found. }  \\
	\multicolumn{1}{r}{\small ``-'':}           & \multicolumn{12}{l}{\small $\Tmax$ is too small to apply the algorithm.}
}

\begin{table}[t]
	\centering
	% \caption{The number of minimal controls discovered by each algorithm within 600 seconds}\label{tab:count-minimal-controls}
	\TABLE{The number of minimal controls discovered within 600 seconds\label{tab:count-minimal-controls}}{
\small
\begin{tabular}{rrc cccc ccc cccc}
\toprule
\ctrlnote \\
\cmidrule(r){1-3}\cmidrule(l){4-14}
Name & Option & $\Tmax$ & S1 & S2 & S3 & S4 & M1 & M2 & M3 & L1 & L2 & L3 & L4 \\
\midrule\midrule
\multicolumn{14}{l}{\textbf{\emph{Finite $\Tmax$}}} \\
\cmidrule(r){1-3}\cmidrule(l){4-14}
\MibS{} & AGG & 1 & \textbf{12} & \textbf{9} & \textbf{18} & \textbf{9} & 113 & 49 & 53 & 26 [3] & 51 & 70 & 538 \\
\BEN{} & DEC & 1 & \textbf{12} & 9 & \textbf{18} & \textbf{9} & 214 & 8 & 81 & 0 & 60 & 47 & 0 \\
\BEN{} & AGG & 1 & \textbf{12} & 9 & \textbf{18} & \textbf{9} & 218 & 2 & 87 & 0 & 60 & 45 & 0 \\
\SEP{} & DEC & 1 & \textbf{12} & \textbf{9} & \textbf{18} & \textbf{9} & \textbf{370} & 82 & 746 & 576 [187] & \textbf{60} & 550 [8] & 81 \\
\SEP{} & AGG & 1 & \textbf{12} & \textbf{9} & \textbf{18} & \textbf{9} & \textbf{370} & 82 & 746 & 575 [187] & \textbf{60} & 550 [8] & 81 \\
\cmidrule(r){1-3}\cmidrule(l){4-14}
\MibS{} & AGG & 3 & 12 & \textbf{9} & 10 & \textbf{9} & 65 & 0 & 2 & 2 & 20 & 13 & 0 \\
\BEN{} & DEC & 3 & \textbf{12} & 9 & \textbf{18} & \textbf{9} & 203 & 6 & 62 & 0 & 60 & 44 & 0 \\
\BEN{} & AGG & 3 & \textbf{12} & 9 & \textbf{18} & \textbf{9} & 185 & 6 & 56 & 0 & 60 & 41 & 0 \\
\SEP{} & DEC & 3 & \textbf{12} & \textbf{9} & \textbf{18} & \textbf{9} & \textbf{340} & 83 & 645 & 962 [360] & \textbf{60} & 509 [2] & 92 \\
\SEP{} & AGG & 3 & \textbf{12} & \textbf{9} & \textbf{18} & \textbf{9} & \textbf{340} & 80 & 641 & 962 [360] & \textbf{60} & 532 [3] & 92 \\
\cmidrule(r){1-3}\cmidrule(l){4-14}
\MibS{} & AGG & 5 & 12 & 19 & 9 & \textbf{9} & 42 & 1 & 2 & 0 & 18 & 0 & 0 \\
\BEN{} & DEC & 5 & \textbf{12} & 31 & \textbf{18} & \textbf{9} & 204 & 6 & 43 & 0 & 60 & 36 & 0 \\
\BEN{} & AGG & 5 & \textbf{12} & 31 & \textbf{18} & \textbf{9} & 188 & 2 & 43 & 0 & 60 & 37 & 0 \\
\SEP{} & DEC & 5 & \textbf{12} & \textbf{31} & \textbf{18} & \textbf{9} & \textbf{344} & 96 & 704 & 797 & \textbf{60} & 618 & 85 \\
\SEP{} & AGG & 5 & \textbf{12} & \textbf{31} & \textbf{18} & \textbf{9} & \textbf{344} & 95 & 660 & 818 & \textbf{60} & 566 & 80 \\
% \cmidrule(r){1-2}\cmidrule(l){3-13}
% \SEP{} & DEC & 15 & \textbf{12} & \textbf{31} & \textbf{18} & \textbf{9} & \textbf{344} & 97 & 634 & 762 & \textbf{60} & 390 & 73 \\
% \SEP{} & AGG & 15 & \textbf{12} & \textbf{31} & \textbf{18} & \textbf{9} & \textbf{344} & 69 & 558 & 641 & \textbf{60} & 236 & 59 \\
% \BEN{} & DEC & 15 & \textbf{12} & 31 & \textbf{18} & \textbf{9} & 204 & 1 & 45 & 0 & 60 & 29 & 0 \\
% \BEN{} & AGG & 15 & \textbf{12} & 31 & \textbf{18} & \textbf{9} & 180 & 2 & 37 & 0 & 60 & 20 & 0 \\
% \MibS{} & AGG & 15 & 9 & 10 & 0 & 9 & 7 & 1 & 0 & 0 & 0 & 0 \\
\midrule\midrule
\multicolumn{14}{l}{\textbf{\emph{Infinite $\Tmax$}}} \\
\cmidrule(r){1-3}\cmidrule(l){4-14}
\PBN{} &  & $\infty$ & 12 & 17 & 10 & \textbf{9} & 70 & 0 & 3 & 0 & 60 & 8 & 0 \\
% \graycmidrules{\cmidrule(r){1-3}\cmidrule(l){4-14}}
\SEP{} & DEC & 1 & \textbf{12} & 9 $\langle 9 \rangle$ & \textbf{18} & \textbf{9} & 370 $\langle 20 \rangle$ & 82 $\langle 1 \rangle$ & 746 $\langle 12 \rangle$ & 576 & \textbf{60} & 550 $\langle 1 \rangle$ & 81 \\
\SEP{} & DEC & 5 & \textbf{12} & \textbf{31} & \textbf{18} & \textbf{9} & \textbf{344} & 96 & 704 & 797 & \textbf{60} & 618 & 85 \\
\SEP{} & DEC & 15 & \textbf{12} & \textbf{31} & \textbf{18} & \textbf{9} & \textbf{344} & 97 & 634 & 762 & \textbf{60} & 390 & 73 \\
\SEP{} & DEC & 60 & \textbf{12} & \textbf{31} & \textbf{18} & \textbf{9} & 344 & 69 & 221 & 124 & \textbf{60} & 118 & 23 \\
\bottomrule
\end{tabular}
	}
		{
		(\textbf{Bold}): the algorithm proved that all minimal controls of size $\leq 7$ were found.

		[Number] / $\langle$Number$\rangle$: The number of non-minimal / infeasible controls.
	}
\end{table}

\Cref{tab:count-minimal-controls} reports the number of minimal controls found within 600 seconds. Numbers in square brackets ([Number]) and angle brackets ($\langle$\text{Number}$\rangle$) denote non-minimal and infeasible controls, respectively.  Bold entries indicate that the algorithm verified that all minimal controls of size at most 7 were found.
The procedure for checking feasibility and minimality is described in \Cref{appendix:verification}. For finite $\Tmax$, \SEP{} may remove some minimal controls because its \tscut{}s rely on the large-$\Tmax$ assumption; hence, non-minimal controls of a larger size may survive in the final output. For the infinite case, too small a $\Tmax$ can yield infeasible controls. \MibS{} missed very few minimal controls due to numerical error, which explains the three non-minimal controls reported for L1.

\subsubsection{Comparison of the total number of controls found.}\mbox{}

\textbf{(i) For finite $\Tmax$ cases:} We compare \MibS{} with the two IBBD algorithms, \BEN{} and \SEP{}.

The two exact methods, \MibS{} and \BEN{}, outperform each other on different instances when $\Tmax$ is small. At $\Tmax=1$, \BEN{} could not prove completeness of the set of minimal controls in S2, while \MibS{} could. \MibS{} found many more solutions than \BEN{} on L1, L3, and especially L4; for L4, \MibS{} is the best method by far, finding 538 controls. Tables in \Cref{appendix:comp-results} reveal that those instances require too many no-good cuts for \BEN{}. As $\Tmax$ increases, however, this tendency disappears and \BEN{} becomes clearly stronger overall.

Although heuristic, \SEP{} remains very effective for finite $\Tmax$: it consistently found the largest numbers of solutions (except L4 for $\Tmax=1$) and scaled better to large instances. When $\Tmax$ is too small, some minimal controls were removed, so non-minimal controls survived as in L1 and L3 for $\Tmax=1$ and $3$. Still, the heuristic is robust in practice: by $\Tmax=5$, it found hundreds of minimal controls without any infeasible control. 
For both \BEN{} and \SEP{}, DEC tends to outperform AGG especially with high $\Tmax$ values and large instances. This suggests that the IBBD framework remains effective even when the number of subproblems is large.

\textbf{(ii) For the infinite $\Tmax$ case:} We compare \PBN{} with \SEP{} used heuristically by setting $\Tmax$ to a large value. We test $\Tmax\in\{1,5,15,60\}$ and report whether any infeasible controls are found. \Cref{tab:count-minimal-controls} shows that \PBN{} is exact but not scalable because it finds relatively few solutions overall and rarely proves the completeness of the search (except S4). In contrast, \SEP{} again serves as a useful heuristic as it finds far more feasible solutions than \PBN{}. In our experiments, $\Tmax=5$ already produces no infeasible controls while providing hundreds of feasible minimal controls even for large instances. Increasing $\Tmax$ will slow down the search process, but it would reduce the risk of finding infeasible controls and potential non-minimal controls if a longer time limit is allowed.

\subsubsection{Comparison on the progression of the search process.}
\begin{figure}
	\caption{Cumulative number of minimal controls over time for $\Tmax=5$}\label{fig:computation-time-t-max-5}
	\includegraphics[width=\linewidth]{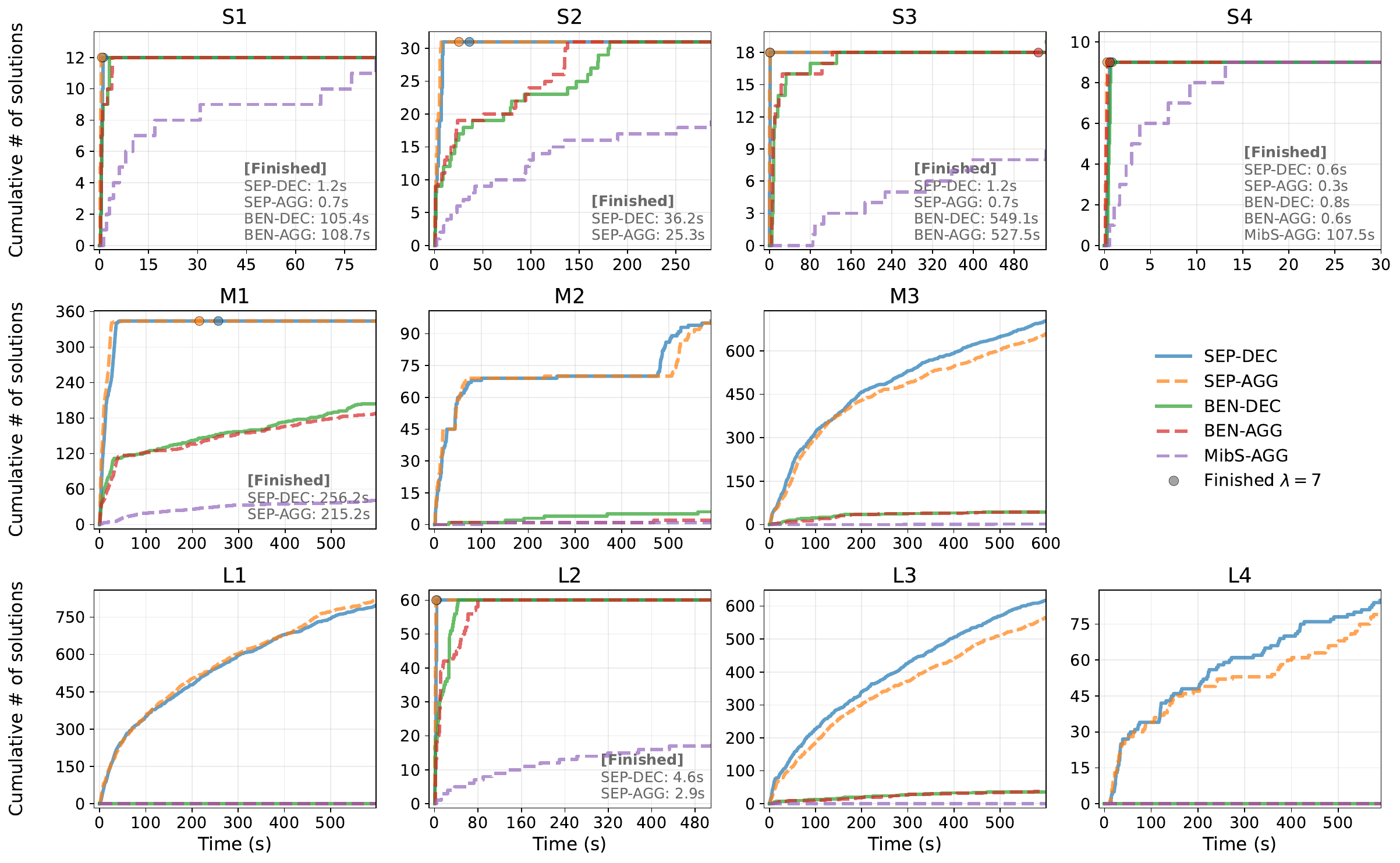}
\end{figure}

\Cref{fig:computation-time-t-max-5} records the cumulative number of minimal controls found over time with $\Tmax=5$, showing how quickly each method makes progress. For algorithms that proved all minimal controls of size $\leq 7$ were found, their completion times are provided. Overall, \SEP{} is the most effective heuristic as it rapidly finds solutions in the early stage and accumulates them steadily even on the harder instances. \SEP{} (DEC) consistently found more solutions than \SEP{} (AGG), especially on difficult medium and large instances (e.g., M2, M3, L1, L3, and L4). However, the aggregated $\LLPModel$ sometimes found more solutions in the very early stage (e.g., M1 and L1) due to its smaller model size and overhead in model building; this observation is much clearer for $\Tmax=60$ as shown in \Cref{appendix:comp-results}.
Therefore, AGG can also be a viable option for time-critical applications.
\MibS{} is generally the weakest method for large $\Tmax$ values, and thus stronger Benders cuts are necessary to solve the problem effectively. \BEN{} often performs competitively on the small instances, but its scalability is limited to a few special cases (M1 and L2). The main benefit of \BEN{} is that it guarantees the minimality of the solutions found, while \SEP{} may not, as we saw in \Cref{tab:count-minimal-controls}.

\subsubsection{Analysis of the strength of Benders cuts.}

The competitiveness of the \SEP{} algorithm stems from its ability to rapidly identify and eliminate infeasible candidates. 
The first half of \Cref{tab:result-benders-statistics} presents the total number of Benders cuts generated by \BEN{} and \SEP{}. \SEP{} could solve S1, S3, S4, and L2 only with dozens of \tscut{}s, whereas \BEN{} required thousands of \attrcut{}s. 
For the remaining instances, employing \tscut{}s either reduced the number of cuts or increased the number of solutions as we observed in \Cref{tab:count-minimal-controls}. 
Furthermore, \tscut{}s are empirically stronger than \attrcut{}s because they contain fewer binary variables, as shown in the second half of \Cref{tab:result-benders-statistics}. Reducing the number of binary variables in a cut can substantially strengthen the restriction imposed on the search, which accelerated \SEP{} compared to \BEN{}.

\newcommand{\newcutstrengthheader}{
\multicolumn{3}{@{}c}{\textbf{Algorithm}} & \multicolumn{11}{c@{}}{\textbf{Instances}} \\ \cmidrule(r){1-3} \cmidrule(l){4-14}
Name & Option & Cuts & S1 & S2 & S3 & S4 & M1 & M2 & M3 & L1 & L2 & L3 & L4
}

\begin{table}[t]
\small
\centering
\TABLE{The statistics of the Benders cuts  $(\Tmax=60)$\label{tab:result-benders-statistics}}{
\begin{tabular}{rrc rrrr rrr rrrr} \toprule
\newcutstrengthheader \\ \midrule
\multicolumn{14}{@{}l}{\textit{\textbf{Total number of cuts}}} \\[0.5em]
\BEN{} & DEC & \attrcut{} & 2126 & 4685 & 5141 & 16 & 3172 & 4217 & 3145 & 2786 & 2594 & 2713 & 2214 \\
\SEP{} & DEC & \tscut{} & 18 & 19 & 11 & 3 & 162 & 134 & 173 & 59 & 18 & 213 & 105 \\
& & \attrcut{} & - & 728 & - & - & 1552 & 2316 & 449 & 135 & - & 724 & 941 \\
\midrule
\multicolumn{14}{@{}l}{\textit{\textbf{Average number of binary variables in a cut}}} \\[0.5em]
\BEN{} & DEC & \attrcut{} & 19.0 & 15.6 & 14.0 & 11.0 & 25.1 & 26.2 & 31.7 & 56.8 & 65.0 & 50.8 & 75.4 \\
\SEP{} & DEC & \tscut{} & 10.9 & 4.1 & 3.4 & 8.0 & 13.3 & 8.0 & 12.9 & 14.6 & 31.3 & 16.0 & 13.3 \\
& & \attrcut{} & - & 18.3 & - & - & 35.8 & 34.4 & 39.1 & 68.2 & - & 67.6 & 105.3 \\
\bottomrule
\end{tabular}}{}
\end{table}

\subsection{Discussion}
We highlight three extensions of the enumerative IBBD algorithm for \TmaxSACP{}:
\begin{itemize}
	\item The heuristic approach \SEP{} for finite $\Tmax$ can be refined by checking subset controls of each returned solution. This would remove surviving non-minimal controls, although it would still not yield an exact method.
	\item We can remove the need for no-good cuts by using the \emph{high point relaxation}, namely the master problem augmented with lower-level feasibility constraints \citep{moore_1990_Mixed}. For $\Tmax=1$, \cite{moon_2022_Bilevel} used this idea in a branch-and-bound implementation.
	\item In some modeling variants, controls inducing no attractor of a prescribed length may be accepted \citep{biane_2026_Why}. \Cref{alg:benders-algorithm} adapts easily to that setting by treating those controls as feasible and omitting no-good cuts.
\end{itemize}

\section{Conclusion}\label{sec:conclusion}

We studied the synchronous attractor control problem for Boolean networks, where the goal is to enumerate all inclusion-wise minimal controls under which every attractor satisfies a given phenotype. To solve this problem, we proposed a bilevel integer programming formulation and developed an infeasibility-based Benders decomposition (IBBD) that derives strong feasibility cuts from forbidden attractors of different lengths.

We also introduced subspace separation, which yields a trap space cut that strengthens one or more attractor cuts under certain conditions. This strengthening is effective when many forbidden attractors are covered by a common fully forbidden trap space, and the computational study shows that the resulting IBBD framework is scalable and practically effective.

Beyond this application, IBBD may be applicable to other bilevel integer programming problems if Benders feasibility cuts can be generated from infeasible solutions of the lower-level problem (see \Cref{appendix:general-ibbd}). 
Our results suggest that IBBD is useful for problems in which coupling constraint violation can be classified into a small number of meaningful categories. If feasible solutions to the original bilevel problem are very sparse, IBBD can be more efficient as it focuses on strong feasibility cuts.

Several directions remain for future work. First, we may develop a method to determine a valid bound on attractor length ($\Tmax$) for a given Boolean network, thereby proving the exactness of trap space cuts.
Another promising direction is to combine the IBBD framework with symbolic methods, such as model checking or constraint programming, to solve the subproblems. As noted in the survey by \cite{hooker_2024_LogicBased}, this hybridization will offer more flexibility in selecting the most suitable method for the subproblems.
Finally, the trap space cut can be applied to controlling other long-term visited states such as asynchronous attractors \citep{rozum_2021_Pystablemotifs,benes_2022_AEON} or minimal trap spaces \citep{pauleve_2023_Marker,riva_2023_Tackling}; we refer to a recent survey by \cite{biane_2026_Why}. This is possible because trap spaces are independent of the update mode \citep{pauleve_2020_Reconciling}.

%\THEEndNotes
% \begingroup \parindent 0pt \parskip 0.0ex \def\enotesize{\normalsize} \theendnotes \endgroup

% Appendix here
% Options are (1) APPENDIX (with or without general title) or
%             (2) APPENDICES (if it has more than one unrelated sections)
% Outcomment the appropriate case if necessary
%
% \begin{APPENDIX}{<Title of the Appendix>}
% \end{APPENDIX}
%
%   or
%
% \begin{APPENDICES}
% \section{<Title of Section A>}
% \section{<Title of Section B>}
% etc
% \end{APPENDICES}

% Acknowledgments here
% \ACKNOWLEDGMENT{The work of KM and KL was supported by the National Research Foundation of Korea (NRF)
% 	grant funded by the Korean government (MSIT, No. NRF-2022R1F1A107414011) and under the framework of international cooperation program managed by National Research Foundation of Korea (No. RS-2023-00259481).
% 	Work of LP was supported by the French Agence Nationale pour la Recherche (ANR) in the scope of the
% 	project ``BNeDiction'' (grant number ANR-20-CE45-0001)
% 	and by Campus France in the scope of the bilateral France-South Korea PHC STAR project number
% 	50147NM.}

% References here (outcomment the appropriate case)

% arXiv mode: inline the pre-built .bbl; otherwise use BibTeX
\ifdefined\arxivmode

\else
  \bibliographystyle{informs2014}
  \bibliography{reference}
\fi

\clearpage
% Restart page numbering for the appendix
\setcounter{page}{1}

\pdfbookmark[0]{Online Appendix}{online-appendix}
\section*{\large Online Supplemental Material for ``A Bilevel Integer Programming Approach for the Synchronous Attractor Control Problem''}
\addcontentsline{toc}{section}{Online Appendix}

\begin{APPENDICES}

	\setcounter{figure}{0}
	\renewcommand{\thefigure}{EC.\arabic{figure}}
	\setcounter{table}{0}
	\renewcommand{\thetable}{EC.\arabic{table}}

	\crefalias{section}{appendix}
	\crefalias{subsection}{appendix}
	\crefalias{subsubsection}{appendix}
	
	\section{Infeasibility-based Benders decomposition for general bilevel IP}\label{appendix:general-ibbd}
	To illustrate the IBBD framework, we consider the following bilevel IP:
	% \paragraph{Bilevel IP}
	\begin{align}
	\text{min }
		& c_w \vbw  \label[objective]{eq:bilevel-example-obj} \\
		\text{s.t. }
		& A \vbw \leq a \label[constraint]{eq:bilevel-example-ulp-1} \\
		& B_w \vbw + B_z \vbz \leq b  \label[constraint]{eq:bilevel-example-ulp-2} \\
		& \vbz \in \argmin_{\tilde{\vbz} \in \mathcal{Z}}
		\{c_z \tilde{\vbz} : D_w \vbw + D_z \tilde{\vbz} \leq d\} \label[constraint]{eq:bilevel-example-llp} \\
		& \vbw \in \mathcal{W}, \quad \vbz \in \mathcal{Z} \label[constraint]{eq:bilevel-example-domain}
	\end{align}
	The upper-level and the lower-level variables are $\vbw$ and $\vbz$, respectively. \Cref{eq:bilevel-example-obj} is the upper-level objective. If lower-level variables appear in the upper-level objective, we reformulate those terms as constraints by introducing an auxiliary upper-level variable and absorb them into \Cref{eq:bilevel-example-ulp-2}. \Cref{eq:bilevel-example-ulp-1} and \Cref{eq:bilevel-example-ulp-2} are the upper-level constraints, where the latter one is the coupling constraint. \Cref{eq:bilevel-example-llp} is the lower-level optimality condition. Finally, \Cref{eq:bilevel-example-domain} defines the domains of $\vbw$ and $\vbz$, which we assume bounded and integral for simplicity.

	We construct the IBBD from \Cref{eq:bilevel-example-obj} and \Crefrange{eq:bilevel-example-ulp-1}{eq:bilevel-example-domain} as follows.
	The \emph{master problem} is obtained by relaxing \Cref{eq:bilevel-example-ulp-2}
	and the lower-level optimality requirement \Cref{eq:bilevel-example-llp}, yielding
	\begin{align}
		\text{min } & c_w \vbw  \\
		\text{s.t. } &  A \vbw \leq a\\
		& \vbw \in \mathcal{W} \\
		& \text{(Feasibility cuts)} 
	\end{align}
	Notably, to prevent lower-level infeasibility due to \Cref{eq:bilevel-example-llp}, high point relaxation \citep{moore_1990_Mixed} may replace the master problem.

	Given a candidate solution $\vbw^*$ from the master problem, the \emph{subproblem} checks whether there exists a lower-level optimal response that satisfies the coupling
	constraints. Specifically, we solve the following lexicographic optimization: 
	\begin{align}
		\operatorname{lex}\min_{\vbz,\Delta} \quad & (c_z \vbz, \ \Delta) \\
		\text{s.t. } & D_w \vbw^* + D_z \vbz \leq d \\
		& B_w \vbw^* + B_z \vbz \leq b + \Delta \cdot \vbone\\
		& \vbz \in \mathcal{Z}, \quad \Delta \in \mathbb{R}_{+}
	\end{align}
	% $\Delta$ to equal to the minimum slack for the coupling constraints. 
	The primary objective ensures lower-level optimality, and the secondary objective then finds the minimum slack $\Delta$ required by any lower-level optimal solution. If the optimal $\Delta$ is positive, then every lower-level optimal solution violates at least one coupling constraint for $\vbw^*$. Then, a problem-dependent feasibility cut should be generated.
	% % FOR PESSIMISTIC
	% 
	% \begin{align}
	% 	\operatorname{lex}\min_{\vbz} \quad & (e \vbz, \ \Delta) \\
	% 	\text{s.t. } & C \vbw^* + D \vbz \leq d \\
	% 	& (B_0 \vbw^* + B \vbz)_i + \Delta \leq b_i  & \quad \forall i = 1, \ldots, m \\
	% 	& (B_0 \vbw^* + B \vbz)_i + \Delta \geq b_i  - M(1 - \alpha_i) & \quad \forall i = 1, \ldots, m \\
	% 	& \sum_{i=1}^m \alpha_i = 1 \\
	% 	& \vbz \in \mathcal{Z}, \quad \Delta \in \mathbb{R}, \quad \boldsymbol{\alpha} \in \{0,1\}^m
	% \end{align}
	% For a given $\vbz$, the binary vector $\boldsymbol{\alpha}$ selects 
	% one coupling constraint, and the constraints force $\Delta$ to equal 
	% its slack while remaining at most the slack of every other constraint. 
	% Hence, $\Delta$ equals the minimum slack across all coupling constraints. 
	% The primary objective ensures lower-level optimality, and the secondary 
	% objective then finds the solution with the smallest minimum slack. 
	% If the optimal $\Delta$ is negative, there exists a lower-level optimal solution that 
	% violates at least one coupling constraint for $\vbw^*$, and a problem-dependent feasibility cut should be generated.
	% 

	\section{Proofs}

	%%%
	%%%

	%%%
	%%%

	\subsection{Proof of \Cref{thm:attractor-characterization}}\label{subsec:proof-attractor-characterization}
	\begin{proof}{Proof of \Cref{thm:attractor-characterization}}
		Suppose a control $\vbd$ satisfying
		\Cref{eq:ulp-exclusivity}. We prove this theorem by showing that the condition $\bigwedge_{j \in \CtrlGeneSet}g_j(\hat{\vbx},\vbd)
			= 0$ is equivalent to the remaining constraints of
		$\SubproblemModelAtT$, which is the if-and-only-if condition for $\hat{\vbx}$ to be an attractor under control $\vbd$.
		Since $(\hat{\vbx},\hat{\vby},\hat{\vbw},\hat{\phiIndicator})$ already satisfy \Crefrange{eq:llp-uncontrollable-1}{eq:llp-domain}, which are independent of $\vbd$, it is 
		sufficient to show the equivalence with the remaining
		\Crefrange{eq:llp-fix}{eq:llp-con-2}.
		Based on the value of $\kj$, we can rephrase the
		\Crefrange{eq:llp-fix}{eq:llp-con-2}
		as \Cref{eq:attractor-condition-rephrase-1}
		\begin{align}
			 & \bigwedge_{j\in \CtrlGeneSet}\left(
			\left(
				d^{\kj}_j \rightarrow \bigwedge_{t \in \Tmaxrange} \sigma(\kj) x_{j,t}
			\right) 
			\wedge \left(
				d^{(1-\kj)}_j \rightarrow \bigwedge_{t \in \Tmaxrange} \sigma(1-\kj) x_{j,t}
				\right)
			\wedge \left(\left(\neg d^{\kj}_j \wedge \neg d^{(1-\kj)}_j\right) \rightarrow \betaj  \right) 
			\right)\label{eq:attractor-condition-rephrase-1}
		\end{align}
		where $\sigma(\kappa)$ denotes the sign function that gives a negation string
		``$\neg$'' if $\kappa = 0$ and an empty string if $\kappa = 1$.
		The first two clauses jointly describe
		\Cref{eq:llp-fix} (controlled case) and the last clause describes \Crefrange{eq:llp-con-1}{eq:llp-con-2} (uncontrolled case).
		Here, we have $\bigwedge_{t \in \Tmaxrange} \sigma(\kj) x_{j,t} = \alphaj$ and
		$\bigwedge_{t \in \Tmaxrange} \sigma(1-\kj) x_{j,t} = 0$ since we already know
		that state $x_{j,1}$ has value $\kj$. Hence, we can
		reformulate \Cref{eq:attractor-condition-rephrase-1} as \Cref{eq:attractor-condition-rephrase-2}
		\begin{align}
			 & 
			\left(\neg d^{\kj}_j \vee \alphaj\right)
			\wedge \left(\neg d^{(1-\kj)}_j\right)
			\wedge \left(d^{\kj}_j \vee d^{(1-\kj)}_j \vee \betaj\right)
			\quad \forall j\in \CtrlGeneSet\label{eq:attractor-condition-rephrase-2}
		\end{align}
		For each $j\in \CtrlGeneSet$, we show the equivalence of
		$g_j(\hat{\vbx},\vbd)$
		in \Crefrange{eq:attractor-characterization-1}{eq:attractor-characterization-3}
		to the terms of \Cref{eq:attractor-condition-rephrase-2} as follows.
		\begin{enumerate}[(i)]
			\item If $\alphaj \wedge \betaj=1$:

			      \Cref{eq:attractor-condition-rephrase-2} can be
			      simplified as $\neg d^{(1-\kj)}_j$. This is true if and only if
			      $g_j(\hat{\vbx},\vbd) = d^{(1-\kj)}_j =0$, as
			      in \Cref{eq:attractor-characterization-1}. This implies that a control must
			      not fix a controllable gene to have its negated value if it has the same value
			      at all times following the transition formula.
			\item If $ \alphaj \wedge\neg \betaj=1$:

			      \Cref{eq:attractor-condition-rephrase-2}
			      can be simplified as $(d^{\kj}_j \vee d^{(1-\kj)}_j) \wedge \neg d^{(1-\kj)}_j
				      = d^{\kj}_j \wedge \neg d^{(1-\kj)}_j$. From the exclusivity constraint,
			      $d^{\kj}_j \rightarrow \neg d^{(1-\kj)}_j$ holds,
			      so $d^{\kj}_j \wedge \neg d^{(1-\kj)}_j$ simplifies to $d^{\kj}_j$
			      and \Cref{eq:attractor-condition-rephrase-2} can be further simplified as $d^{\kj}_j$.
			      Therefore, \Cref{eq:attractor-condition-rephrase-2} is true if
			      and only if $ g_j(\hat{\vbx},\vbd) = (1- d^{\kj}_j) = 0$, as
			      in \Cref{eq:attractor-characterization-2}. This implies that a control must
			      fix a controllable gene to have the current value if it has the same value at
			      all times but does not follow the transition formula.
			\item If $\neg \alphaj \wedge \betaj=1$:

			      \Cref{eq:attractor-condition-rephrase-2}
			      can be simplified as $\neg d^{\kj}_j \wedge \neg d^{(1-\kj)}_j$. This is true
			      if and only if $g_j(\hat{\vbx},\vbd) = d^0_j+d^1_j = 0$, as
			      in \Cref{eq:attractor-characterization-3}. This implies that a control must
			      not be applied if a controllable gene can take both 0 and 1 over time.
			\item If $\neg \alphaj \wedge\neg \betaj=1$:

			    Since $(1-\alphaj)(1-\betaj)=0$ for every attractor solution, this case never arises and can be disregarded.
		\end{enumerate}
		This completes the proof.
		\hfill\Halmos
	\end{proof}

	%%%
	%%%
	\subsection{Proof of \Cref{thm:attractor-cut}}\label{subsec:proof-attractor-cut}
	\begin{proof}{Proof of \Cref{thm:attractor-cut}}
		Since $\SubproblemAt{\hat{\vbd}}{T}$ has the optimal solution $(\hat{\vbx},\hat{\vby},\hat{\vbw},0)$ with the objective value of 0, phenotype condition in \Cref{eq:ulp-p-hard} is violated. Therefore, the same solution should not be feasible under any feasible solution of $\ULPModel$.
		According to \Cref{thm:attractor-characterization}, this condition is the same
		as $\sum_{j \in \CtrlGeneSet} g_j(\hat{\vbx},\vbd) \geq 1$, which is indeed
		identical to \Cref{eq:attractor-cut} due to the following reason.
		Converting the
		case definitions in \Crefrange{eq:attractor-characterization-1}{eq:attractor-characterization-3},
		$\sum_{j \in \CtrlGeneSet} g_j(\hat{\vbx},\vbd) \geq 1$ is equivalent to
		\begin{align*}
			 & \sum_{j \in \CtrlGeneSet} \left( \alphaj\betaj  d^{(1-\kj)}_j
			+ (1-\alphaj)\betaj(d^{\kj}_j+d^{(1-\kj)}_j)  + \alphaj(1-\betaj) (1-d^{\kj}_j)                                                \right) \geq 1            \\
		 & \Leftrightarrow \sum_{j \in \CtrlGeneSet} \left( \betaj  d^{(1-\kj)}_j
		+ (1-\alphaj)\betaj d^{\kj}_j  + \alphaj(1-\betaj) (1-d^{\kj}_j)  \right) \geq 1 \\
		& (\because (1-\alphaj)(1-\betaj)=0 \text{ for any attractor solution}) \notag \\
			 & \Leftrightarrow \sum_{j \in \CtrlGeneSet} \left( \betaj  d^{(1-\kj)}_j
			+ (1-\alphaj) d^{\kj}_j 
			+ (1-\betaj) (1-d^{\kj}_j)  \right) \geq 1
		\end{align*}
		which can be simplified as \Cref{eq:attractor-cut}. It is the if-and-only-if condition for a control that does not induce $\hat{\vbx}$ as an attractor; therefore, the infeasible candidate $\hat{\vbd}$ is removed by this constraint. This completes the proof.
		\hfill\Halmos
	\end{proof}
	\subsection{Proof of \Cref{thm:trap-space-cut}}\label{subsec:proof-trap-space-cut}
	\begin{proof}{Proof of \Cref{thm:trap-space-cut}}
		Suppose $\vbd$ is a feasible solution to $\ULPModel$ for some $\TargetSize$ and $\PrevControlSet$, but it does not satisfy \Cref{eq:trap-space-cut}. This implies the left-hand side is 0. Then, $u^k_j \leq d^k_j$ for all $j\in \CtrlGeneSet, k \in \B$ because $1-d^{k}_j$ in the first term must be 0 whenever $u^{k}_j=1$. In the second term, $(d^{(1-k)}_j=1) \rightarrow (d^k_j=0) \rightarrow (u^k_j=0)$ due to \Cref{eq:ulp-exclusivity} and the previous result. Therefore, for the second term to be 0, $(d^{(1-k)}_j=1) \rightarrow (h^k_j=0)$ must hold, which is the contrapositive of condition \Cref{eq:trap-space-condition} in \Cref{thm:control-preserving-trap-spaces}.
		Then, we can apply \Cref{thm:control-preserving-trap-spaces} to deduce that
		$\vbh$ is still a trap space under control $\vbd$.
		Moreover, since $\phigene \in \UnctrlGeneSet$, the phenotype evaluation is unaffected by any control, so $x_{\phigene}=0$ still holds for every state belonging to $\vbh$. Hence $\vbh$ remains a fully forbidden trap space under control $\vbd$.
		Since every trap space includes at least one attractor, a forbidden attractor must appear as an
		optimal solution of $\SubproblemModelAtT$ for some $T\in \Tmaxrange$
		assuming that $\Tmax$ is sufficiently large. Therefore, $\vbd$ cannot
		be a feasible solution to $\ULPModel$, leading to a
		contradiction of the initial hypothesis. This property does not depend on $\TargetSize$ and
		$\PrevControlSet$, and thus \Cref{eq:trap-space-cut} is valid for all feasible
		solutions of $\ULPModel$ for all $\TargetSize$ and $\PrevControlSet$.
		\hfill\Halmos
	\end{proof}

	\subsection{Proof of \Cref{thm:stronger-benders-cut}}\label{subsec:proof-stronger-benders-cut}
	%%%
	%%%

	\begin{proof}{Proof of \Cref{thm:stronger-benders-cut}.}
		\Cref{eq:trap-space-cut} and \Cref{eq:attractor-cut} 
        % \kds{consist of single terms that are either binary variables or their negations.} 
        can be decomposed into terms consisting of either a single variable or its negation. We show that each term of \Cref{eq:trap-space-cut} corresponding to $j \in \CtrlGeneSet$ implies the one for \Cref{eq:attractor-cut}, which is equivalent to
		\begin{align}
			 & \sum_{k \in \B} \left(u^k_j(1- d^{k}_j) + (1-u^k_j) h^k_j d^{(1-k)}_j \right) \leq \left((1-\alphaj) d^{\kj}_j
			+ \betaj  d^{(1-\kj)}_j
			+  (1-\betaj) (1-d^{\kj}_j)  \right). \label{eq:cut-strength-inequality}
		\end{align}
		Three cases may happen.
		\begin{enumerate}[(i)]
			\item If $\exists k\in \B$ such that
			      $(u^k_j,u^{(1-k)}_j)=(h^k_j,h^{(1-k)}_j)=(1,0)$:

			      In this case, $\alphaj=1$ and $\kj=k$ because all values of $x_{j,t}$
			      are fixed to be $k$ by trap space $\vbh$. The left-hand side gives
			      $(1-d^k_j)$ and the right-hand side gives $(\betaj d^{(1-k)}_j +
				      (1-\betaj)(1-d^k_j))$. Since $u^k_j=1$, the hypothesis of
			      \Cref{thm:stronger-benders-cut} implies $\betaj=0$. Therefore, the
			      right-hand side reduces to $(1-d^k_j)$.
			      Hence, the left-hand side equals the right-hand side.
			\item If $\exists k\in \B$ such that $(u^k_j,u^{(1-k)}_j)=(0,0)$ and
			      $(h^k_j,h^{(1-k)}_j)=(1,0)$:

			      Similar to case (i), $\alphaj=1$ and $\kj=k$ hold. Since $u^k_j=u^{(1-k)}_j=0$,
			      gene $j$ is uncontrolled in $\SubproblemAt{\vbu}{T}$, so
			      \Crefrange{eq:llp-con-1}{eq:llp-con-2} force gene $j$ to
			      follow the transition formula, giving $\betaj=1$. The left-hand side gives
			      $d^{(1-k)}_j$ and the right-hand side gives $\betaj d^{(1-k)}_j = d^{(1-k)}_j$.
			      Hence, the left-hand side equals the right-hand side.
			\item Otherwise:

			      \Crefrange{eq:separation-IP-exclusivity-1}{eq:separation-IP-control}
			      imply $(u^k_j,u^{(1-k)}_j)=(h^k_j,h^{(1-k)}_j)=(0,0)$
			      for all remaining cases. The left-hand side becomes 0, and thus it is less than
			      or equal to the right-hand side.
		\end{enumerate}
		Hence, the \tscut{} \Cref{eq:trap-space-cut} implies the \attrcut{} \Cref{eq:attractor-cut} derived with $\vbx$.
		\hfill\Halmos
	\end{proof}

	%%%
	%%%

	%%%
	%%%
	\section{Counterexample: a trap space cut may not be stronger than an attractor cut}\label{subsec:counterexample}
	Consider control
	$\vbu$ and the fully forbidden trap space $\vbh$ given by the assignment
	$(u^0_1,u^1_1,u^0_2,u^1_2)=(h^0_1,h^1_1,h^0_2,h^1_2)=(0,0,0,1).$
	In this example, $(u^0_j \vee u^1_j)\rightarrow \neg \betaj$ does not hold, so \Cref{thm:stronger-benders-cut} does not apply.
	The corresponding trap space cut is $(1-d^1_2) \geq 1$ while the attractor $\{(11)\}$
	belonging to $\vbh$ gives attractor cut $(d^0_1+d^0_2 \geq 1)$ because
	all $\alphaj$ and $\betaj$ are 1. Candidate control $\vbd=\vbzero$
	satisfies the trap space cut, but it violates the attractor cut.
\begin{figure}[t]

\FIGURE{
\subfloat[Inputs\label{subfig:example-cut-strength-transition-formula}]{%
	\parbox[b]{0.30\linewidth}{
		\centering
		\SingleSpacedXI
		\small
		\begin{tabular}{r@{}l}
			\toprule
			\multicolumn{2}{l}{\textit{\textbf{The transition formulas}}} \\
			\midrule
			$f_1(\vbx)$ & \, $= x_1$ \\
			$f_2(\vbx)$ & \, $= (\neg x_1 \vee x_2) \wedge (x_1 \vee \neg x_2)$ \\
			\midrule
			\multicolumn{2}{l}{\textit{\textbf{The phenotype}}} \\
			\midrule
			$\phiFunction(\vbx)$ & \, $= \neg x_2$ \\
			\bottomrule
		\end{tabular}
		\vspace{0pt}
	}
}\hfill
\subfloat[STG (no control)\label{subfig:example-cut-strength-stg}]{%
	\parbox[b]{0.25\linewidth}{
		\centering
		\begin{tikzpicture}[scale=1]
			\node at (0,0) [attractor_fill] {};
			\node at (2,0) [attractor_fill] {};
			\node at (0,2) [attractor_fill] {};
			\node at (2,2) [attractor_fill] {};
			\draw[] (0,0) node(00) {00};
			\draw[] (2,0) node(10) {10};
			\draw[] (0,2) node(01) {01};
			\draw[] (2,2) node(11) {11};
			\draw[densely dashed,gray,thick] (0-\exgapL,0-\exgapL) rectangle (0+\exgapL,2+\exgapL);
			\draw[densely dashed,gray,thick] (2-\exgapL,2-\exgapL) rectangle (2+\exgapL,2+\exgapL);
			\draw[densely dashed,gray,thick] (2-\exgapL,0-\exgapL) rectangle (2+\exgapL,0+\exgapL);
			\draw[thick_arc] (00) .. controls +(right:1) and +(right:1) .. (01);
			\draw[thick_arc] (01) .. controls +(left:1) and +(left:1) .. (00);
			\draw[thick_arc] (10) +(up:\exgap) arc (-100:230:2.5mm);
			\draw[thick_arc] (11) +(up:\exgap) arc (-100:230:2.5mm);
		\end{tikzpicture}
	}
}\hfill
\subfloat[STG (control $\vbu$ fixing $x_2=1$)\label{subfig:example-cut-strength-stg-control}]{%
	\parbox[b]{0.26\linewidth}{
		\centering
		\begin{tikzpicture}[scale=1]
			\node at (0,2) [attractor_fill] {};
			\node at (2,2) [attractor_fill] {};
			\draw[] (0,0) node(00) {00};
			\draw[] (2,0) node(10) {10};
			\draw[] (0,2) node(01) {01};
			\draw[] (2,2) node(11) {11};
			\draw[densely dashed,gray,thick] (0-\exgapL,2-\exgapL) rectangle (0+\exgapL,2+\exgapL);
			\draw[densely dashed,gray,thick] (2-\exgapL,2-\exgapL) rectangle (2+\exgapL,2+\exgapL);
			\draw[densely dashed,orange,thick] (0-\exgapLL,2-\exgapLL) rectangle (2+\exgapLL,2+\exgapLL);
			\draw[normal_arc] (00) -- (01);
			\draw[thick_arc] (01) +(up:\exgap) arc (-100:230:2.5mm);
			\draw[normal_arc] (10) -- (11);
			\draw[thick_arc] (11) +(up:\exgap) arc (-100:230:2.5mm);
			\draw[densely dashed,gray,thick,draw opacity=0.0] (2-\exgapL,0-\exgapL) rectangle (2+\exgapL,0+\exgapL);
		\end{tikzpicture}
	}
}\hfill
\begin{minipage}[b]{0.14\linewidth}
	\vspace{0pt}
	\centering
	\small
	\SingleSpacedXII
	\begin{tikzpicture}
		\SingleSpacedXII
		\small
		\begin{scope}[local bounding box=legendBox]
			% Attractor
			\node[legendBoxStyle, fill=gray!30, text=black] at (0,0) {Attractor};
			% Trap space
			\node[legendBoxStyle, dashed, gray, text=black] at (0,-1) {Trap space};
			% Control
			\node[legendBoxStyle, dashed, orange, text=black] at (0,-2) {Control};
		\end{scope}
		% Surrounding gray box
		% \node[draw, gray!50, fit=(legendBox), inner sep=8pt] {};
	\end{tikzpicture}
\end{minipage}
}
{A counterexample showing that the trap space cut may not be stronger than the attractor cut.\label{fig:cut-strength-counterexample}}{}

\end{figure}

\section{Computational results}\label{appendix:comp-results}

\subsection{Benchmark instances}

\begin{table}[t]
	\centering
	\TABLE{Summary of data sets\label{tab:data_summary}}{
		\begin{tabular}{l rrrr rrr rrrr} \toprule
			                             & \multicolumn{4}{c}{\textbf{Small}} & \multicolumn{3}{c}{\textbf{Medium}} & \multicolumn{4}{c}{\textbf{Large}}                                             \\ \cmidrule(lr){2-5} \cmidrule(lr){6-8} \cmidrule(lr){9-12}
			                             & S1                                 & S2                                  & S3                                 & S4 & M1 & M2  & M3  & L1  & L2 & L3 & L4  \\ \midrule
			\multicolumn{12}{l}{\emph{\textbf{Input statistics}}} \\[0.2em]
			Total number of genes $\numgenes$                  & 20                                 & 17                                  & 18                                 & 13 & 28 & 32  & 35  & 59  & 66 & 53 & 75  \\
			Number of controllable genes $\numctrlgenes$              & 19                                 & 15                                  & 13                                 & 11 & 25 & 26  & 31  & 56  & 60 & 50 & 73  \\
			Number of clauses in $\ClauseSet^1$            & 39                                 & 25                                  & 33                                 & 21 & 45 & 120 & 103 & 112 & 87 & 93 & 116 \\
			Number of clauses in $\ClauseSet^0$            & 34                                 & 21                                  & 25                                 & 17 & 35 & 75  & 70  & 96  & 84 & 92 & 109 \\
			\midrule
			\multicolumn{12}{l}{\emph{\textbf{Maximum forbidden attractor length}}} \\[0.2em]
			Found by an IP solver (3600s)$^{\dagger}$ & \textbf{3} & \textbf{4} & \textbf{2} & \textbf{12} & 16 & 13 & 44 & 36 &45 & 48 & 40 \\
			Observed during experiments & 1                                  & 4                                   & 1                                  & 1  & 4  & 6   & 20  & 4   & 1  & 40 & 19  \\ 
			\bottomrule
		\end{tabular}}{
		$^{\dagger}$Obtained from the IP model \eqref{ip:max-forbidden-attractor-obj}--\eqref{ip:max-forbidden-attractor-domain} with $\Tmax=100$ and control-size bounded above as $\TargetSizeMax=7$. 
		
		Bold values indicate the optimum was found.
	}
\end{table}

\Cref{tab:data_summary} summarizes the eleven benchmark instances used in our experiments.
Columns $\numgenes$ and $\numctrlgenes$ represent the total numbers
of genes and controllable genes, respectively. Columns $|\ClauseSet^1|$
and $|\ClauseSet^0|$ represent the total number of clauses in the CNFs of
$f_i$ and $\neg f_i$, respectively. We also report the maximum forbidden attractor lengths observed empirically during the experiments. The first row of this part is obtained by solving an IP model described in \Cref{appendix:verification} with a time limit of 3600 seconds. Although those values seem to be large, the actual maximum forbidden attractor lengths observed during the experiments are much smaller as the second row indicates. Those long forbidden attractors are rare, which supports the effectiveness of the IBBD framework.

\subsection{Verification of feasibility, minimality, and maximum forbidden attractor length}\label{appendix:verification}

\paragraph{For finite-$\Tmax$:}
We independently re-check the feasibility of each reported control by solving the corresponding subproblems for all lengths up to $\Tmax=100$. To verify minimality, for every reported control $\vbd$ we enumerate all strict subsets of $\vbd$ and test each subset in the same way; if any subset is feasible, then $\vbd$ is classified as non-minimal. 

\paragraph{For infinite-$\Tmax$:}
Feasibility is verified with model checking code, which was adapted from \cite{cifuentes-fontanals_2022_Control}. This method exactly determines whether the controlled network admits a forbidden attractor regardless of the length. This procedure is more expensive than solving the subproblems for finite $\Tmax$; hence, we did not check minimality for infinite-$\Tmax$ cases.

To reduce repeated work, each checked control is stored with two cached results:
(i) the minimum length of any attractor, (ii) the minimum length of an attractor violating the phenotype condition. Using these results, we can immediately determine the feasibility of any control for any $\Tmax$. 

\paragraph{Checking the maximum forbidden attractor length via IP:}
To jointly determine the maximum forbidden attractor length under a particular control, we use the following IP model.
In these runs, we set the maximum control size to $\TargetSizeMax=7$.

\begin{align}
	\max_{\vbd,\vbx,\vby,\vbw,\phiIndicator,\mathbf{q},\boldsymbol{\delta}} \quad & \sum_{t\in\Tmaxrange} t\cdot w_t \label[objective]{ip:max-forbidden-attractor-obj}\\
	\text{s.t. \quad }& \text{(Lower-level) \Crefrange{eq:llp-fix}{eq:llp-domain}} \notag \\
	& \text{(Upper-level) \Cref{eq:ulp-exclusivity}}\notag \\
	& \phiIndicator = 0 \label[constraint]{ip:max-forbidden-attractor}\\
	& \textstyle \sum_{j\in\CtrlGeneSet}\sum_{k\in\B} d_j^k \le \TargetSizeMax
	\label[constraint]{ip:max-forbidden-attractor-control-size}\\
	& q_t = \textstyle\sum_{r=t+1}^{\Tmax} w_r
	&& \forall t\in \Tmaxrange\setminus\{\Tmax\}
	\label[constraint]{ip:max-forbidden-attractor-q-def}\\
	& \delta_{i,t}\ge x_{i,t}-x_{i,1}, \quad
	  \delta_{i,t}\ge x_{i,1}-x_{i,t}
	&& \forall i\in\GeneSet,\ t\in \Tmaxrange
	\label[constraint]{ip:max-forbidden-attractor-delta-lb}\\
	& \delta_{i,t}\le x_{i,t}+x_{i,1}, \quad
	  \delta_{i,t}\le 2-x_{i,t}-x_{i,1}
	&& \forall i\in\GeneSet,\ t\in \Tmaxrange
	\label[constraint]{ip:max-forbidden-attractor-delta-ub}\\
	& \textstyle \sum_{i\in\GeneSet}\delta_{i,t}\ge q_{t-1}
	&& \forall t\in \Tmaxrange\setminus\{1\}
	\label[constraint]{ip:max-forbidden-attractor-subcycle}\\
	& \mathrlap{
		\mathbf{q}\in\B^{\Tmax-1}, \qquad
		\boldsymbol{\delta} \in\B^{\numgenes\times\Tmax}}
	\label[constraint]{ip:max-forbidden-attractor-domain}
\end{align}

Here, $q_t=1$ indicates that the selected attractor length exceeds $t$, and
$\delta_{i,t}=1$ indicates that the state of gene $i$ at time $t$ differs from
that at time $1$. \Cref{ip:max-forbidden-attractor-q-def}
encodes $q_t=\sum_{r=t+1}^{\Tmax} w_r$, and
\Crefrange{ip:max-forbidden-attractor-delta-lb}{ip:max-forbidden-attractor-subcycle}
prevent an earlier repetition of the first state. Hence, if $w_t=1$, the model
selects an attractor whose exact length is $t$, not merely a multiple of a
shorter cycle.

\clearpage
\subsection{Full experiment results}
% \subsubsection{The progression of the search process}\mbox{}

\begin{figure}[ht]
	\caption{Cumulative number of minimal controls over time for $\Tmax=1$}\label{fig:computation-time-t-max-1}
	\includegraphics[width=0.9\linewidth]{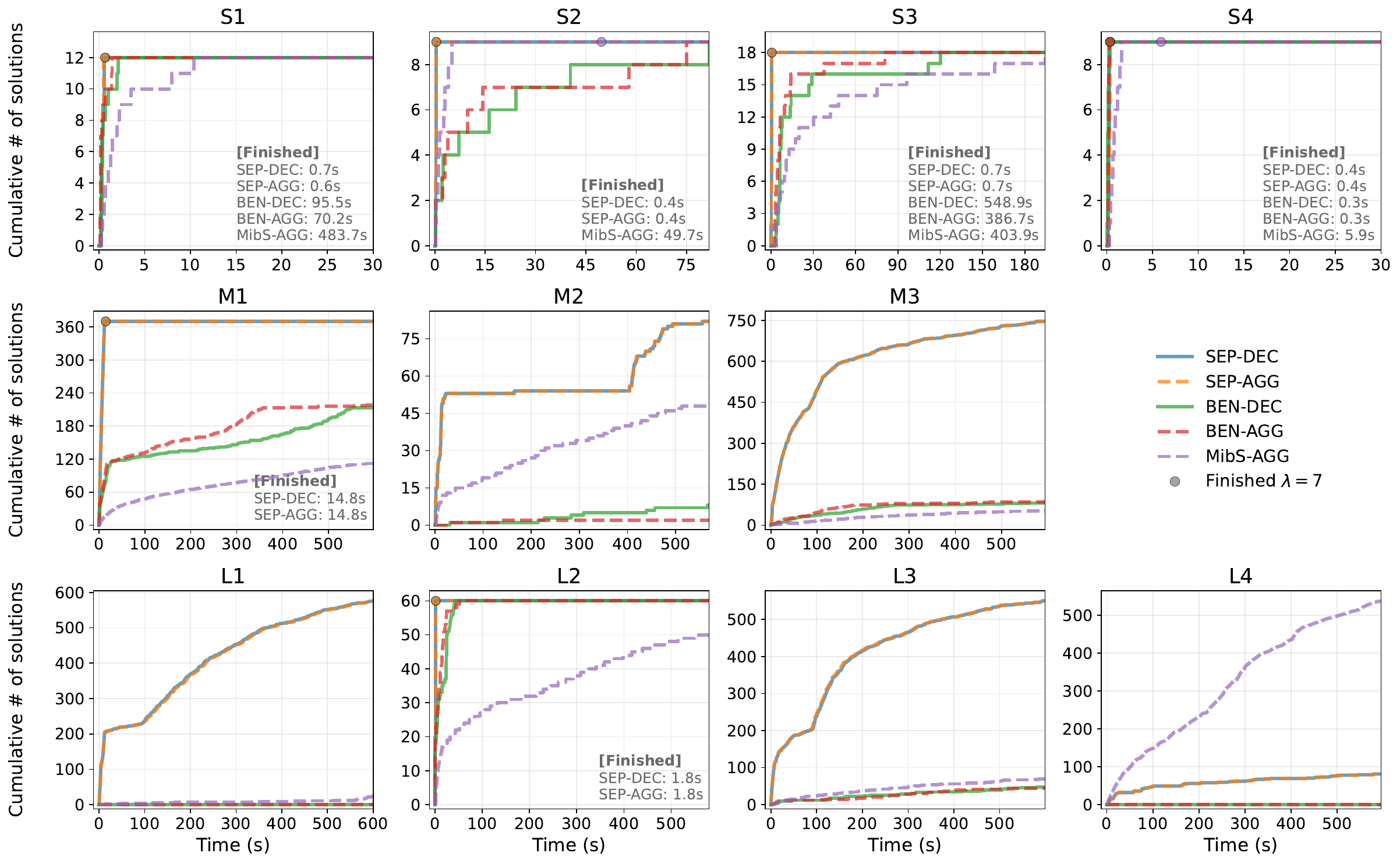}
\end{figure}

\begin{figure}[ht]
	\caption{Cumulative number of minimal controls over time for $\Tmax=3$}\label{fig:computation-time-t-max-3}
	\includegraphics[width=0.9\linewidth]{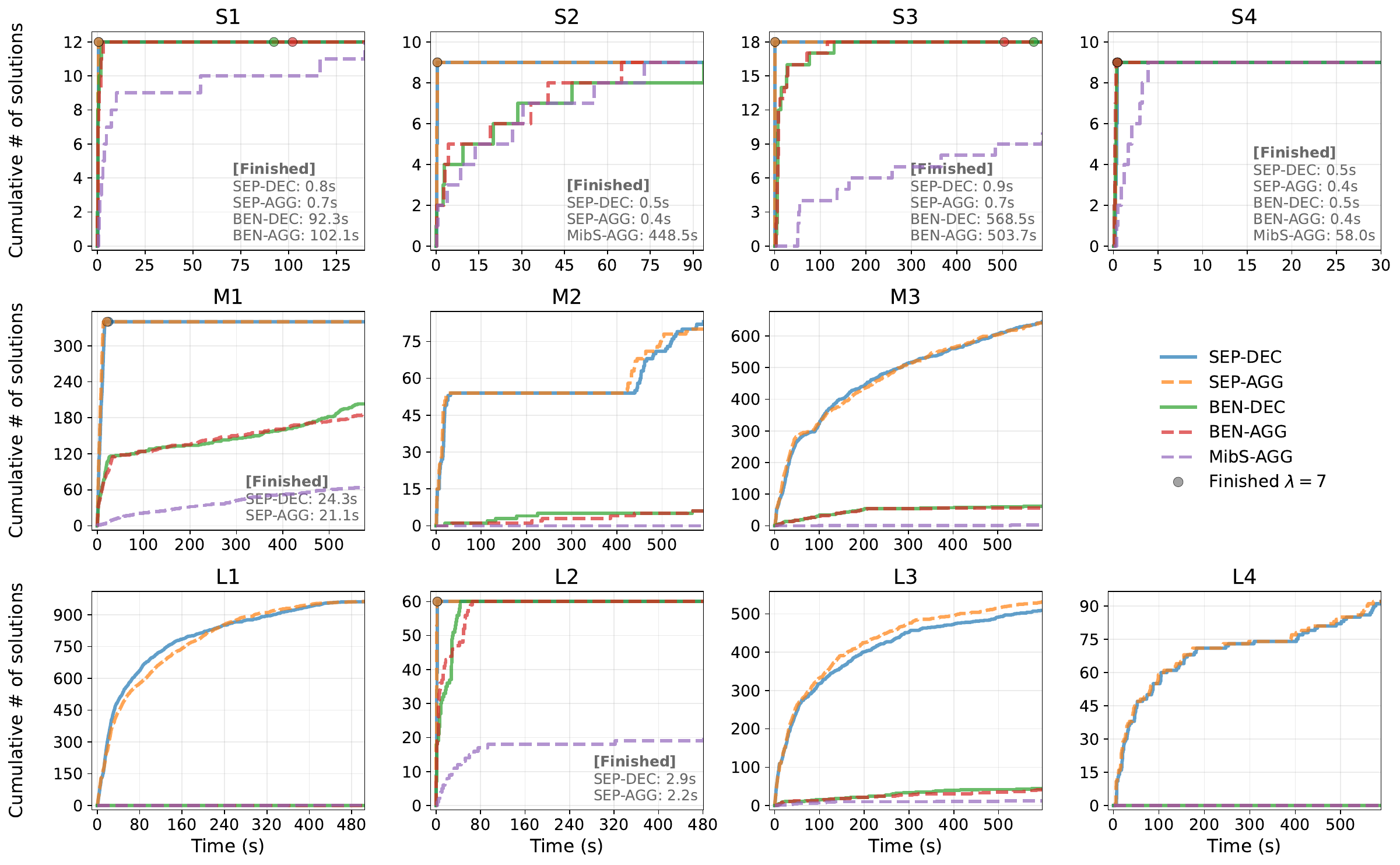}
\end{figure}

\begin{figure}[ht]
	\caption{Cumulative number of minimal controls over time for $\Tmax=15$}\label{fig:computation-time-t-max-15}
	\includegraphics[width=0.9\linewidth]{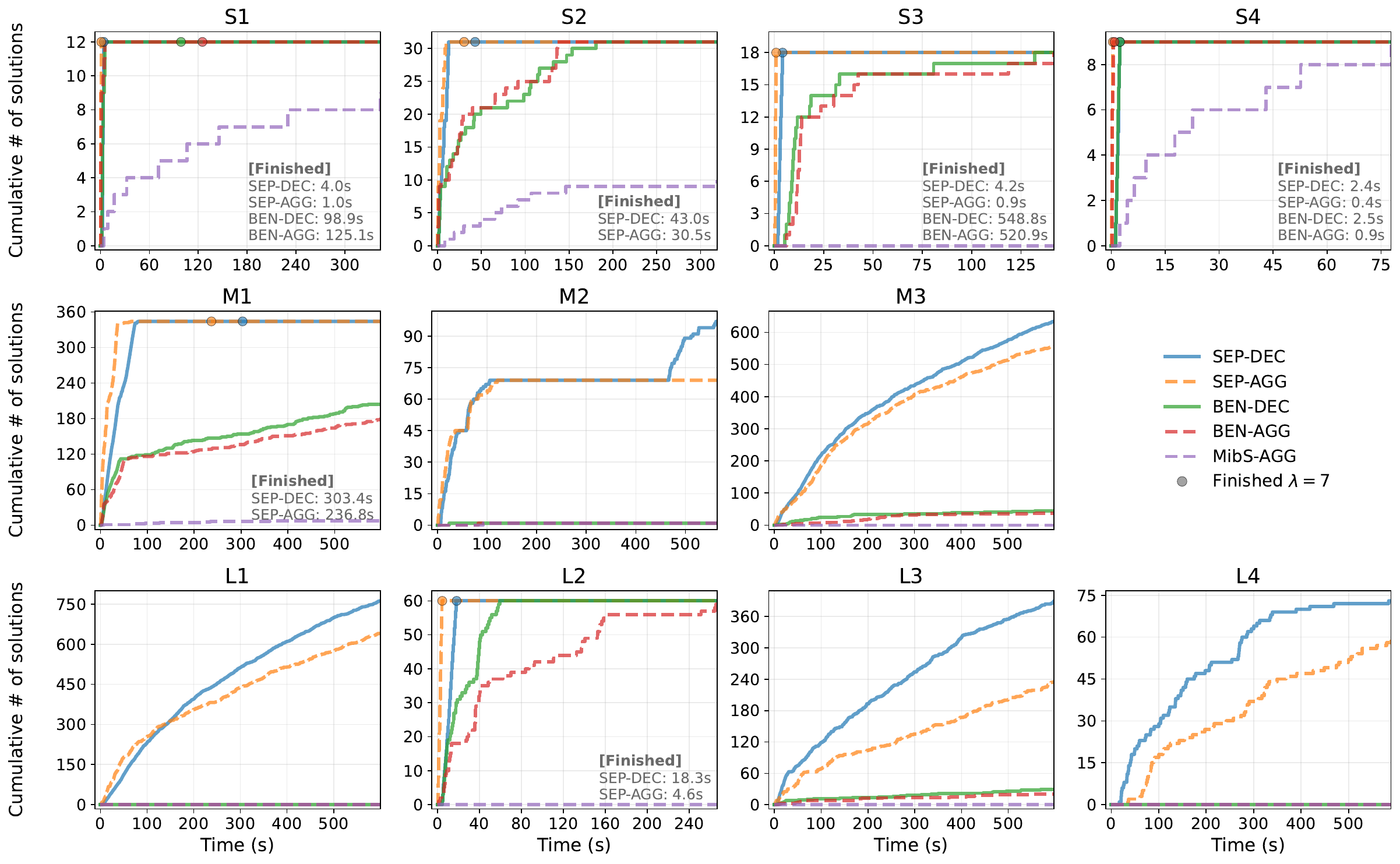}
\end{figure}

% \begin{figure}[ht]
% 	\caption{Cumulative number of minimal controls over time for $\Tmax=15$}\label{fig:computation-time-t-max-15}
% 	\includegraphics[width=0.9\linewidth]{agg_comptime_plot_T15.pdf}
% \end{figure}

\begin{figure}[ht]
	\caption{Cumulative number of minimal controls over time for $\Tmax=60$}\label{fig:computation-time-t-max-60}
	\includegraphics[width=0.9\linewidth]{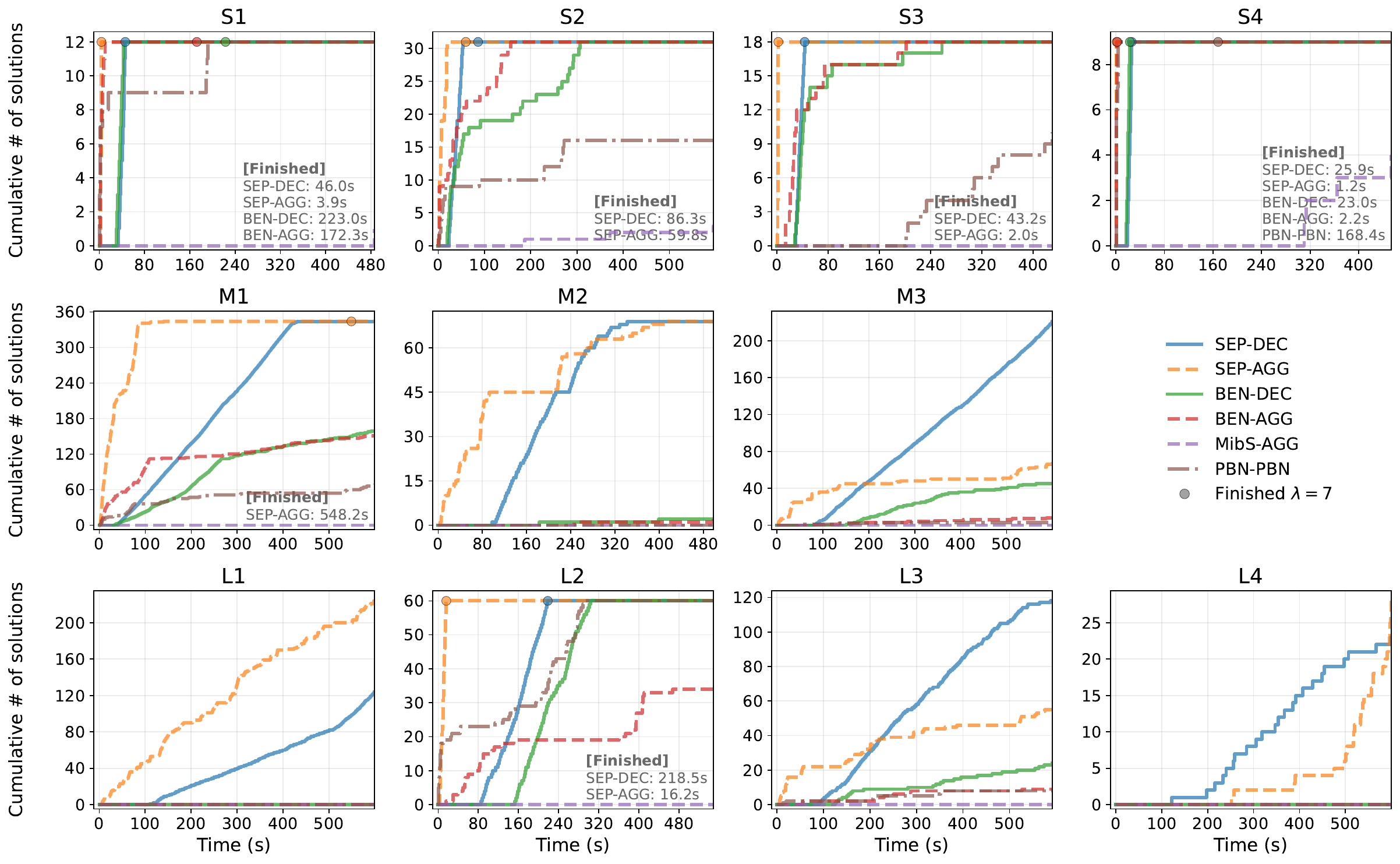}
\end{figure}

% \clearpage
% \subsubsection{The statistics of Benders cuts}\mbox{}

\begin{table}[ht]
\small
\centering
\TABLE{The statistics of the Benders cuts\label{tab:result-benders-statistics-merged}}{
\begin{tabular}{rrc rrrr rrr rrrr} \toprule
\newcutstrengthheader \\ \midrule
\multicolumn{14}{@{}l}{\textit{\textbf{$(\Tmax=1)$ Total number of cuts}}} \\[0.5em]
\BEN{} & DEC & \attrcut{} & 2126 & 1496 & 5198 & 16 & 3949 & 4563 & 2625 & 1345 & 3187 & 2919 & 2167 \\
& & No-good cut & - & 3547 & - & - & - & 197 & 1019 & 1016 & - & 253 & 384 \\
\SEP{} & DEC & \tscut{} & 18 & 5 & 11 & 3 & 148 & 119 & 219 & 144 & 18 & 254 & 133 \\
& & No-good cut & - & - & - & - & - & 2993 & 2793 & 1568 & - & 1834 & 1476 \\
\midrule
\multicolumn{14}{@{}l}{\textit{\textbf{$(\Tmax=1)$ Average number of binary variables in a cut}}} \\[0.5em]
\BEN{} & DEC & \attrcut{} & 19.0 & 15.0 & 14.0 & 11.0 & 25.0 & 26.0 & 31.0 & 56.0 & 65.0 & 50.0 & 73.0 \\
& & No-good cut & - & 30.0 & - & - & - & 52.0 & 62.0 & 112.0 & - & 100.0 & 146.0 \\
\SEP{} & DEC & \tscut{} & 10.9 & 4.6 & 3.4 & 8.0 & 13.6 & 7.6 & 12.9 & 17.9 & 31.3 & 16.7 & 14.9 \\
& & No-good cut & - & - & - & - & - & 52.0 & 62.0 & 112.0 & - & 100.0 & 146.0 \\
\midrule\midrule
\multicolumn{14}{@{}l}{\textit{\textbf{$(\Tmax=3)$ Total number of cuts}}} \\[0.5em]
\BEN{} & DEC & \attrcut{} & 2126 & 1481 & 5202 & 16 & 3820 & 4482 & 3115 & 1741 & 3407 & 3184 & 2270 \\
& & No-good cut & - & 3299 & - & - & - & 183 & 582 & 834 & - & 200 & 415 \\
\SEP{} & DEC & \tscut{} & 18 & 5 & 11 & 3 & 153 & 123 & 248 & 120 & 18 & 262 & 134 \\
& & \attrcut{} & - & - & - & - & 108 & 2 & 1168 & 34 & - & 82 & 4 \\
& & No-good cut & - & - & - & - & - & 2721 & 1277 & 1382 & - & 1783 & 1382 \\
\midrule
\multicolumn{14}{@{}l}{\textit{\textbf{$(\Tmax=3)$ Average number of binary variables in a cut}}} \\[0.5em]
\BEN{} & DEC & \attrcut{} & 19.0 & 15.0 & 14.0 & 11.0 & 25.0 & 26.0 & 31.3 & 56.1 & 65.0 & 50.1 & 73.0 \\
& & No-good cut & - & 30.0 & - & - & - & 52.0 & 62.0 & 112.0 & - & 100.0 & 146.0 \\
\SEP{} & DEC & \tscut{} & 10.9 & 4.6 & 3.4 & 8.0 & 13.0 & 7.6 & 13.1 & 16.6 & 31.3 & 16.7 & 14.8 \\
& & \attrcut{} & - & - & - & - & 35.8 & 31.0 & 38.5 & 60.1 & - & 64.9 & 79.0 \\
& & No-good cut & - & - & - & - & - & 52.0 & 62.0 & 112.0 & - & 100.0 & 146.0 \\
\midrule\midrule
\multicolumn{14}{@{}l}{\textit{\textbf{$(\Tmax=5)$ Total number of cuts}}} \\[0.5em]
\BEN{} & DEC & \attrcut{} & 2128 & 5098 & 5202 & 16 & 3867 & 4745 & 4017 & 3115 & 3284 & 3253 & 2486 \\
& & No-good cut & - & - & - & - & - & - & - & - & - & 1 & 54 \\
\SEP{} & DEC & \tscut{} & 18 & 19 & 11 & 3 & 163 & 137 & 269 & 129 & 18 & 316 & 138 \\
& & \attrcut{} & - & 728 & - & - & 1786 & 3057 & 2497 & 1827 & - & 1716 & 766 \\
& & No-good cut & - & - & - & - & - & 64 & 148 & - & - & 132 & 663 \\
\midrule
\multicolumn{14}{@{}l}{\textit{\textbf{$(\Tmax=5)$ Average number of binary variables in a cut}}} \\[0.5em]
\BEN{} & DEC & \attrcut{} & 19.0 & 15.6 & 14.0 & 11.0 & 25.1 & 26.2 & 31.8 & 56.8 & 65.0 & 50.6 & 74.7 \\
& & No-good cut & - & - & - & - & - & - & - & - & - & 100.0 & 146.0 \\
\SEP{} & DEC & \tscut{} & 10.9 & 4.1 & 3.4 & 8.0 & 13.3 & 8.0 & 13.2 & 17.6 & 31.3 & 16.9 & 14.6 \\
& & \attrcut{} & - & 18.3 & - & - & 35.7 & 34.3 & 39.3 & 67.3 & - & 67.3 & 99.7 \\
& & No-good cut & - & - & - & - & - & 52.0 & 62.0 & - & - & 100.0 & 146.0 \\
\midrule\midrule
\multicolumn{14}{@{}l}{\textit{\textbf{$(\Tmax=15)$ Total number of cuts}}} \\[0.5em]
\BEN{} & DEC & \attrcut{} & 2128 & 5117 & 5206 & 16 & 3817 & 4856 & 4021 & 3139 & 3282 & 3264 & 2426 \\
& & No-good cut & - & - & - & - & - & - & - & - & - & - & 2 \\
\SEP{} & DEC & \tscut{} & 18 & 19 & 11 & 3 & 163 & 140 & 255 & 137 & 18 & 277 & 118 \\
& & \attrcut{} & - & 728 & - & - & 1786 & 3180 & 2387 & 1630 & - & 1796 & 1160 \\
& & No-good cut & - & - & - & - & - & - & 17 & - & - & 2 & 20 \\
\midrule
\multicolumn{14}{@{}l}{\textit{\textbf{$(\Tmax=15)$ Average number of binary variables in a cut}}} \\[0.5em]
\BEN{} & DEC & \attrcut{} & 19.0 & 15.7 & 14.0 & 11.0 & 25.1 & 26.2 & 31.8 & 56.8 & 65.0 & 50.7 & 75.4 \\
& & No-good cut & - & - & - & - & - & - & - & - & - & - & 146.0 \\
\SEP{} & DEC & \tscut{} & 10.9 & 4.1 & 3.4 & 8.0 & 13.3 & 7.9 & 13.0 & 17.7 & 31.3 & 16.2 & 14.2 \\
& & \attrcut{} & - & 18.3 & - & - & 35.7 & 34.3 & 39.7 & 67.0 & - & 68.4 & 105.2 \\
& & No-good cut & - & - & - & - & - & - & 62.0 & - & - & 100.0 & 146.0 \\
\bottomrule
\end{tabular}}{}
\end{table}

\end{APPENDICES}

%%%%%%%%%%%%%%%%%
\end{document}